\documentclass{amsart}

\usepackage{amsmath,amssymb}
\usepackage{amsthm}
\usepackage{mathtools}

\renewcommand{\Re}{\operatorname{Re}}
\renewcommand{\Im}{\operatorname{Im}}
\DeclareMathOperator{\sgn}{sgn}
\renewcommand{\arctan}{\operatorname{arctg}}

\DeclarePairedDelimiterX\innerp[2]{\langle}{\rangle}{#1, #2}
\DeclarePairedDelimiterX\ccint[2]{\lbrack}{\rbrack}{#1, #2}

\newtheorem{theorem}{Theorem}
\newtheorem{lemma}{Lemma}
\newtheorem{claim}{Claim}
\newtheorem{corollary}{Corollary}

\author{Petar Melentijevi\'c}
\thanks{The author is partially supported by MPNTR grant 174017, Serbia}
\title[Hollenbeck-Verbitsky conjecture]{Hollenbeck-Verbitsky conjecture on best constant inequalities for analytic and co-analytic projections}

\begin{document}
\begin{abstract}
  In this paper we address the problem of finding the best constants in inequalities of the form:
   $$ \|\big(|P_+f|^s+|P_-f|^s\big)^{\frac{1}{s}}\|_{L^p({\mathbb{T}})}\leq A_{p,s} \|f\|_{L^p({\mathbb{T}})},$$
  where $P_+f$ and $P_-f$ denote analytic and co-analytic projection of a complex-valued function $f \in L^p({\mathbb{T}}),$ for $p \geq 2$ and all $s>0$, thus proving Hollenbeck-Verbitsky conjecture from \cite{HV.OTAA}. We also prove the same inequalities for\\ $1<p\leq\frac{4}{3}$ and $s\leq \sec^2\frac{\pi}{2p}$ and confirm that $s=\sec^2\frac{\pi}{2p}$ is the sharp cutoff for $s.$ The proof uses a method of plurisubharmonic minorants and an approach of proving the appropriate "elementary" inequalities that seems to be new in this topic. 
  
  We show that this result implies best constants inequalities for the projections on the real-line and half-space multipliers on $\mathbb{R}^n$ and an analog for analytic martingales. A remark on an isoperimetric inequality for harmonic functions in the unit disk is also given. 
 \end{abstract}
\maketitle

\section{\textbf{Introduction}}

\subsection{$H^p$ spaces and analytic and co-analytic projections}
Let $\mathbf{D}=\{z\in \mathbf{C}:|z|<1\}$ be the unit disk  and let $\mathbf{T}=\{z\in \mathbf{C}:|z|=1\}$  be the  unit circle. For  $0<p\le\infty$
we denote by $L^p(\mathbf{T})$ the  Lebesgue space  on $\mathbf{T}$. The space $H^p(\mathbf{T})$ contains  all $f \in  L^p(\mathbf{T})$   for
which all negative Fourier coefficients are equal to zero, i.e.,
\begin{equation*}
\hat{f}(n)  =  \frac{1}{2\pi}\int_{0}^{2\pi} f(e^{it})e^{-int} dt = 0\quad \text{for every  integer}\  n<0.
\end{equation*}
The Riesz projection  operator $P_+:L^p(\mathbf{T})\to H^p(\mathbf{T})$,  and the  co-analytic projection operator  $P_- = I- P_+$,  where $I$ is the
identity operator on  $L^p(\mathbf{T})$ and $\quad f(\zeta)= \sum_{n =   -\infty}^{+\infty} \hat{f}(n) \zeta^n,$ are defined as
\begin{equation}
\label{projekcije}
P_+f (\zeta)  =  \sum_{n=0}^{+\infty} \hat{f}(n)  \zeta^n\quad\text{and}\quad   P_- f(\zeta)=\sum_{n= - \infty}^{-1} \hat{f}(n) \zeta^n.
\end{equation}

For a function  $f$ on $\mathbf{D}$ and $r\in (0,1)$ we denote by $f_r$ the function  $f_r (\zeta) =  f(r\zeta)$, $\zeta\in\overline{\mathbf{D}}$. The
harmonic Hardy space $h^p$, $0<p\le \infty$, consists of all  harmonic complex-valued  functions  $f$  on $\mathbf{D}$ for  which the integral mean
\begin{equation*}
M_p(f,r)  = \left\{ \int_\mathbf{T} |f_r(\zeta)|^p\frac{|d\zeta|}{2\pi}\right\}^{\frac 1p}
\end{equation*}
remains bounded as $r$ approaches $1$.  Since $|f|^p$, $1\le p<\infty$ is  subharmonic on  $\mathbf{D}$, the integral mean $M_p(f,r)$ is increasing in
$r$. The norm on $h^p$ in this case is given by
\begin{equation*}
\|f\|_p = \lim_{r\rightarrow 1} M_p(f,r).
\end{equation*}
The analytic Hardy space $H^p$ is the subspace of $h^p$ that contains analytic functions. For the  theory of  Hardy spaces in the unit disk we refer to
classical books \cite{DUREN.BOOK, GARNETT.BOOK, KOOSIS} and  to the more recent Pavlovi\'{c} book \cite{PAVLOVIC.BOOK} for new results on the Hardy  space theory in the  unit disc. Since $|f|^p$ is subharmonic for every $0<p<\infty$ if $f$ is holomorphic, the (quasi)norm on the Hardy space $H^p$ may be introduced as
\begin{equation*}
\|f\|_p  = \lim_{ r \rightarrow 1}  M_p(f,r).
\end{equation*}
For $p=\infty$ the space $H^\infty$ contains all bounded analytic functions in  the unit disc.

It  is well known that for $f\in H^p$ the radial boundary value $f^{\ast}(\zeta)= \lim_{r\rightarrow 1-} f(r\zeta)$ exists for almost every
$\zeta\in \mathbf{T}$,  $f^\ast$ belongs to the space  $H^p(\mathbf{T})$, and we have the isometry relation $\|f^{\ast}\|_{L^p(\mathbf{T})} =
\|f\|_{H^p}$. Moreover, given $\varphi \in H^p (\mathbf{T})$, the Cauchy (and the Poisson)  extension  of it  gives  a function  $f$ in  $H^p$ such
that   $f^\ast (\zeta ) = \varphi(\zeta)$ for almost every $\zeta\in \mathbf{T}$. Therefore, one may identify the spaces $H^p(\mathbf{T})$  and  $H^p$
via the   isometry $f\rightarrow f^\ast$. Let us say, also, that after identifying these two spaces, the Riesz operator  $P_+$
may be represented as the Cauchy integral (with density $f^\ast$)
\begin{equation*}
P_+ f (z) =  \frac{1}{2\pi i} \int_{\mathbf{T}} \frac{f^\ast(\zeta)}{\zeta - z} d\zeta, \quad z \in \mathbf{D}.
\end{equation*}

The Riesz operator is not bounded on $L^1(\mathbf{T})$; actually there does not exist a bounded projection of $L^1(\mathbf{T})$ onto
$H^1(\mathbf{T})$. This was shown by Newman (\cite{NEWMAN}) and generalized by Rudin (\cite{RUDIN}). From the Parseval's identity it follows that the operator $P_+$ is an
orthogonal projection of  $L^2(\mathbf{T})$ onto $H^2(\mathbf{T})$. Therefore, the norm of $P_+$ as an operator from  ${L^2(\mathbf{T})}$
onto   $H^2(\mathbf{T})$ is equal to  $1$. If $p=\infty$, then the Riesz projection operator  $P_+$  is also  unbounded.
The question for $p\ne 2$ is more subtle. It is a classical result that the Riesz  projection $P_+$ is bounded if it is considered as an operator
on $L^p(\mathbf{T})$  for every $1<p<\infty$, \cite{RIESZ}. There are many proofs of this fact; for example, a proof that uses Hardy-Stein identity can be found in \cite{PAVLOVIC.BOOK}. However, finding the accurate norm of this operator has turned out to be a more delicate problem. 
Some partial results on the norm of the Riesz operator $P_+$  were obtained by Gohberg and Krupnik. It was done in the case when $p=2^n$,  $n\in\mathbf{N}$). They also made a conjecture on its value that has been proved by Hollenbeck and Verbitsky, \cite{HV.JFA}. Many other results that preceded the final solution of Gohberg-Krupnik conjecture can be found in \cite{GOHBERGKRUPNIK,GK2, GK3,KRUPNIK,KRVERB1,KRVERB2,PAPADOPOULOS,VERBITSKY.ISSLED}. Best constants inequalities concerning Hilbert transform, Riesz projection or some other singular integral operators using purely analytic or probabilistic methods are considered by many authors, let us mention also \cite{BAERNSTEIN, BANUELOSWANG, BURKHOLDER, DAVIS, DINGGRAFAKOSZHU, HKV, JANAKIRAMAN, OSEKOWSKI,  PICHORIDES.STUDIA, TOMASZEWSKI}.

In \cite{HV.JFA}, Hollenbeck and Verbitsky proved slightly stronger result than that stated in Gohberg - Krupnik conjecture:
\begin{equation}
\label{eq:HV12}
\|\max \{ |P_ + (f)|,  |P_- (f)|\} \|_{L^p  (\mathbf{T})}\le \frac  1 {\sin  \frac  \pi p} \|f\|_{L^p (\mathbf{T})},  \quad  f\in {L^p (\mathbf{T})}.
\end{equation}
The proof uses the method of plurisubharmonic minorants. In fact, they proved the "elementary" inequality:
\begin{equation}
\label{eq:HVEL1}
\max\{|z|^p,|w|^p\}\leq \frac{1}{\sin^p\frac{\pi}{p}}|z+\overline{w}|^p-E(z,w),
\end{equation}
where $E(z,w)=\frac{2|\cos\frac{\pi}{p}|^{1-\frac{p}{2}}}{\sin\frac{\pi}{p}}\Re(zw)^{\frac{p}{2}}, 1<p\leq 2$ is plurisubharmonic. Then, integrating this inequality with $z=P_+f(\zeta)$ and $w=\overline{P_-f(\zeta)}$ over $\zeta \in \mathbb{T}$ the desired inequality is established. This implies the result on the norm of the Riesz projection on the circle for $1<p\leq 2,$ while the conclusion for $p>2$ follows from the duality arguments. The paper \cite{HV.JFA} also contains the analogous result for the real line and certain weighted $L^p$ spaces on $\mathbb{R}.$

In \cite{HV.OTAA}, an inequality similar to \eqref{eq:HVEL1} is proved for $p>2,$ with the different choice of $E(z,w).$ This concluded the proof of the inequality \eqref{eq:HV12} for $p>2.$ 

\subsection{Statement of the problem and recent progress} In the same paper, \cite{HV.OTAA}, the authors posed a problem of finding the optimal constant $A _{p,s }$ in the inequality
\begin{equation}
\label{eq:HVconj}
\|  ( |P_ +  f | ^s  +  |P_-  f |^s)^{\frac {1}{s}}\|_{L^p (\mathbf{T})}\le A_{p,s } \|f\|_{L^p (\mathbf{T})}
\end{equation}
for $ f\in {L^p (\mathbf{T})}, \,    1<p<\infty, \, 0<s<\infty.$ They conjectured that
\begin{equation}
\label{eq:C12}
A _ {p,s } = \frac {2^{\frac{1}{s}}}{2\cos\frac{\pi}{2p}},  \quad   1<p<2,\  0<s<\sec^2\frac{\pi} {2p}
\end{equation}
and
\begin{equation}
\label{eq:C2b}
A _{p,s } = \frac{2^{\frac {1}{s}}}{2\sin\frac{\pi}{2p}},  \quad  2<p<\infty, \ 0<s<\csc^2\frac \pi {2p }.
\end{equation}
In terms of the above notation, Hollenbeck-Verbitsky results from \cite{HV.JFA} and \cite{HV.OTAA} reads as 
$$A_{p,\infty}  = \frac{1}{\sin\frac{\pi}{p}},\quad   1<p<\infty.$$ 
It is an easy consequence of Verbitsky's result from \cite{VERBITSKY.ISSLED} that the best constant in the inequality (\ref{eq:HVconj}) for real-valued functions is given by (\ref{eq:C12}) and (\ref{eq:C2b}) for all positive $s.$ However, for $s> \max\{\sec^2\frac{\pi}{2p},\csc^2\frac{\pi}{2p}\}$ in the complex case the same does not hold, i.e. $A_{p,s}$ depends in a more substantial way on $s.$ 

Let us describe the motivation of introducing the conjectural value $A_{p,s}$ and the existence of cutoff's for $s.$ For the family $f_{\gamma}(z)=\alpha\Re g_{\gamma}(z)+\imath \beta\Im g_{\gamma}(z)$, where $g_{\gamma}(z)=\bigl(\frac{1+z}{1-z}\bigr)^{\gamma}$, we easily find
$$P_+f_{\gamma}(z)=\frac{\alpha+\beta}{2}g_{\gamma}(z)+\frac{\alpha-\beta}{2}\quad\text{and}\quad P_-f_{\gamma}(z)=\frac{\alpha-\beta}{2}\big(\overline{g_{\gamma}}(z)-1\big).$$ 
Thus,
$$\frac{\|\big(|P_+f_{\gamma}|^s+|P_-f_{\gamma}|^s\big)^{\frac{1}{s}}\|_{L^p({\mathbb{T}})}}{\|f_{\gamma}\|_{L^p({\mathbb{T}})}}= \frac{\|\big(|\frac{\alpha+\beta}{2}g_{\gamma}+\frac{\alpha-\beta}{2}|^s+|\frac{\alpha-\beta}{2}\overline{g_{\gamma}}-\frac{\alpha-\beta}{2}|^s\big)^{\frac{1}{s}}\|_{L^p({\mathbb{T}})}}{\|\frac{\alpha+\beta}{2}g_{\gamma}+\frac{\alpha-\beta}{2}\overline{g_{\gamma}}\|_{L^p({\mathbb{T}})}}.$$

For $s\geq 1,$ by Minkowski inequalities for sequences and integrals, respectively, we find:
\begin{align*}
&\bigg\|\bigg(\bigg|\frac{\alpha+\beta}{2}g_{\gamma}+\frac{\alpha-\beta}{2}\bigg|^s+\bigg|\frac{\alpha-\beta}{2}\overline{g_{\gamma}}-\frac{\alpha-\beta}{2}\bigg|^s\bigg)^{\frac{1}{s}}\bigg\|_{L^p({\mathbb{T}})}\\
&\geq \bigg\|\bigg(\bigg|\frac{\alpha+\beta}{2}g_{\gamma}\bigg|^s+\bigg|\frac{\alpha-\beta}{2}\overline{g_{\gamma}}\bigg|^s\bigg)^{\frac{1}{s}}-2^{\frac{1}{s}-1}\big|\alpha-\beta\big|\bigg\|_{L^p({\mathbb{T}})}\\
&\geq \bigg\|\bigg(\bigg|\frac{\alpha+\beta}{2}g_{\gamma}\bigg|^s+\bigg|\frac{\alpha-\beta}{2}\overline{g_{\gamma}}\bigg|^s\bigg)^{\frac{1}{s}}\bigg\|_{L^p({\mathbb{T}})}-2^{\frac{1}{s}-1}\big|\alpha-\beta\big|.
\end{align*}

For $s<1,$ we have $\frac{p}{s}>1,$ hence, using the inequality $|x\pm y|^s\geq |x|^s-|y|^s$ and Minkowski integral inequality with exponent $\frac{p}{s}$ we get:
\begin{align*}
&\bigg\|\bigg(\bigg|\frac{\alpha+\beta}{2}g_{\gamma}+\frac{\alpha-\beta}{2}\bigg|^s+\bigg|\frac{\alpha-\beta}{2}\overline{g_{\gamma}}-\frac{\alpha-\beta}{2}\bigg|^s\bigg)^{\frac{1}{s}}\bigg\|_{L^p({\mathbb{T}})}\\
&\geq \bigg\|\bigg|\frac{\alpha+\beta}{2}g_{\gamma}\bigg|^s+\bigg|\frac{\alpha-\beta}{2}\overline{g_{\gamma}}\bigg|^s-2^{1-s}\big|\alpha-\beta\big|^s\bigg\|_{L^{\frac{p}{s}}({\mathbb{T}})}^{\frac{1}{s}}\\
&\geq \bigg[\bigg\|\bigg|\frac{\alpha+\beta}{2}g_{\gamma}\bigg|^s+\bigg|\frac{\alpha-\beta}{2}\overline{g_{\gamma}}\bigg|^s\bigg\|_{L^{\frac{p}{s}}({\mathbb{T}})}-2^{1-s}\big|\alpha-\beta\big|^s\bigg]^{\frac{1}{s}}.
\end{align*}

Similarly we obtain the analogous inequalities with the sign $"\leq"$ changing the signs in front of $2^{\frac{1}{s}-1}|\alpha-\beta|$ and $2^{1-s}|\alpha-\beta|^s$ from the minus to the plus sign. 
By the fact that $\lim_{\gamma\rightarrow \frac{1}{p}}\|g_{\gamma}\|_{L^p(\mathbb{T})}=+\infty,$ we arrive at
$$\lim_{\gamma\rightarrow \frac{1}{p}}\frac{\|\big(|P_+f_{\gamma}|^s+|P_-f_{\gamma}|^s\big)^{\frac{1}{s}}\|_{L^p({\mathbb{T}})}}{\|f_{\gamma}\|_{L^p({\mathbb{T}})}}=\lim_{\gamma\rightarrow \frac{1}{p}}\frac{\|\big(|\frac{\alpha+\beta}{2}g_{\gamma}|^s+|\frac{\alpha-\beta}{2}\overline{g_{\gamma}}|^s\big)^{\frac{1}{s}}\|_{L^p({\mathbb{T}})}}{\|\frac{\alpha+\beta}{2}g_{\gamma}+\frac{\alpha-\beta}{2}\overline{g_{\gamma}}\|_{L^p({\mathbb{T}})}},$$
which is, by $|\Im g_{\gamma}|=\tan\frac{\pi\gamma}{2}\Re g_{\gamma},$ equal to $\frac{\big(|\alpha+\beta|^s+|\alpha-\beta|^s\big)^{\frac{1}{s}}}{2\big(\alpha^2\cos^2\frac{\pi}{2p}+\beta^2\sin^2\frac{\pi}{2p}\big)^\frac{1}{2}}.$\\
Putting $\alpha+\beta=1, \alpha-\beta=e^{-y},$ $y\geq 0,$ we get:
\begin{equation}
\label{sharpness}
A_{p,s}\geq \max_{y\geq 0}\frac{2^\frac{1}{s}\cosh^{\frac{1}{s}}\frac{s y}{2}}{\big(\cosh y\pm \cos\frac{\pi}{p}\big)^{\frac{1}{2}}},
\end{equation}
where we choose the plus sign for $1<p\leq 2$ and the minus sign for $p\geq 2.$ More details on these estimates from below and calculations of this maximum can be found in \cite{MELENTIJEVICMARKOVIC} or later on in Lemma \ref{lemaoKfunkciji}. Here, as we will see,  the range of $0<s\leq \max\{\sec^2\frac{\pi}{2p},\csc^2\frac{\pi}{2p}\}$, i.e. equalities (\ref{eq:C12}) and (\ref{eq:C2b}) corresponds to the case when the above maximum is achieved for $y=0.$ 

Also, note that, since $\max\{a,b\}=\lim_{s \rightarrow \infty}\big(a^s+b^s\big)^{\frac{1}{s}},$ for $a,b>0,$ it is expected that $\lim_{s \rightarrow +\infty}A_{p.s}=A_{p,\infty}$ for all $p>1.$ This is proved in \cite{MELENTIJEVICMARKOVIC} using \eqref{eq:HVEL1} and its analog for $p>2$ with explicit calculations of the lower bound for $A_{p,s}$.  

In the paper \cite{KALAJ.TAMS}, Kalaj proved that $A_{p,2}=\frac{1}{\sqrt{2}\sin\frac{\pi}{2p'}},$ where $p'=\max\{p,\frac{p}{p-1}\},$ confirming Hollenbeck-Verbitsky conjecture in case $s=2$. He also proved the reverse estimate to \eqref{eq:HVconj} for $s=2$ i.e. the best constant inequality where the left- and right-hand sides of \eqref{eq:HVconj} are replaced by each other. The application of these inequalities to the isoperimetric inequality for harmonic functions are also given. We will see at the end of this paper that better estimates can be proved without using them. 

The author of this paper found an approach of proving the appropriate "elementary" inequalities used in the proof of \eqref{eq:HVconj} and its reversed counterparts for $s=2$, in \cite{MELENTIJEVIC.PHD}, which significantly shorten their proofs. The idea is to analyze stationary points of the appropriate function inside a rectangle and reducing the proof to checking the non-negativity of this function on its sides. It has turned out that this method works for $s\leq \max\{p,\frac{p}{p-1}\}$ and in \cite{MELENTIJEVICMARKOVIC} the conjecture is proved for $p\geq 2$ and $s \leq p$ and $1<p \leq \frac{5}{4}$ and $s=4.$ Also, $s\leq \max\{p,\frac{p}{p-1}\}$ is the maximum range of $s$ where this approach can be applied.

Using the monotonicity of the power means, it is easy to see that if \eqref{eq:C12} or \eqref{eq:C2b} is proved for some $s_0 \leq \max\{\sec^2\frac{\pi}{2p},\csc^2\frac{\pi}{2p}\},$ then it holds for $s \in (0,s_0]$(see  also \cite{MELENTIJEVICMARKOVIC}). Hence to confirm \eqref{eq:C12} and \eqref{eq:C2b} it is enough to prove it for the case $s= \max\{\sec^2\frac{\pi}{2p},\csc^2\frac{\pi}{2p}\}.$ This is the largest value for which $A_{p,s}$ has the form as \eqref{eq:C12} and \eqref{eq:C2b}, depending on whether $1<p\leq 2$ of $p>2,$ by \cite{MELENTIJEVICMARKOVIC}. We will call it critical value and, by this observation, concentrate only on the range $s\geq\max\{\sec^2\frac{\pi}{2p},\csc^2\frac{\pi}{2p}\}.$

\section{\textbf{Main results and corollaries}}

In this paper, we will prove \eqref{eq:C12} for $1<p\leq \frac{4}{3}$ and \eqref{eq:C2b} for all $p\geq 2.$ Therefore, the main result reads as follows:
\begin{theorem}
	\label{HVcp}
	Let $p \geq 2$ and $s>0$, $f \in L^p(\mathbb{T})$ and $P_+f$ and $P_-f$ its analytic and co-analytic projections. Then the following sharp inequality holds
	\begin{equation}
	\label{eq:opstanorma2}
	\|\big(|P_+f|^s+|P_-f|^s\big)^{\frac{1}{s}}\|_{L^p({\mathbb{T}})}\leq A_{p,s} \|f\|_{L^p({\mathbb{T}})},
	\end{equation}
	where $A_{p,s}=2^{\frac{1}{s}}\max\limits_{y \in \mathbb{R}^+}\frac{\cosh^{\frac{1}{s}}\frac{s y}{2}}{\sqrt{\cosh y-\cos\frac{\pi}{p}}}.$
	Specially, for $s\leq \csc^2\frac{\pi}{2p}$ we have:
	\begin{equation}
	\label{eq:norm2b}
	\|\big(|P_+f|^s+|P_-f|^s\big)^{\frac{1}{s}}\|_{L^p({\mathbb{T}})}\leq \frac{2^{\frac{1}{s}}}{2\sin\frac{\pi}{2p}} \|f\|_{L^p({\mathbb{T}})}.
	\end{equation}
	
	If $1<p\leq\frac{4}{3}$ and $s\leq\sec^2\frac{\pi}{2p}$, then we have:
	\begin{equation}
	\label{eq:norm12}
	\|\big(|P_+f|^s+|P_-f|^s\big)^{\frac{1}{s}}\|_{L^p({\mathbb{T}})}\leq \frac{2^{\frac{1}{s}}}{2\cos\frac{\pi}{2p}} \|f\|_{L^p({\mathbb{T}})}.
	\end{equation}	
\end{theorem}

The sharpness of these constants is clear from (\ref{sharpness}). Also, note that the main result from \cite{HV.OTAA} can be inferred from our Theorem 1 and hence, by duality the value of the exact norm of Riesz projection from \cite{HV.JFA}.

Using main results from this theorem as granted, we will now mention several corollaries of them. 
\begin{corollary}
	\label{corollary1}
	Let $f$ be an $L^p(\mathbb{T})$ function and $\tilde{f}(\zeta)=-\imath \sum_{n=-\infty}^{+\infty}(\sgn n)\hat{f}(n)\zeta^n$ its harmonic conjugate. Then the following sharp estimate holds:
	\begin{equation}
	\label{fiftilda}
	 \bigg\|\bigg(\bigg|\frac{f+\imath \tilde{f}}{2}\bigg|^s+\bigg|\frac{f-\imath \tilde{f}}{2}\bigg|^s\bigg)^{\frac{1}{s}}\bigg\|_{L^p(\mathbb{T})} \leq A_{p,s} \|f\|_{L^p(\mathbb{T})}
	\end{equation}
	with the same constant $A_{p,s}$ for the values of $p$ and $s$ as in Theorem \ref{HVcp}.
	\end{corollary}	
\begin{proof}
Using inequalities (\ref{eq:opsta2b}) and (\ref{eq:main12}) with $z=\frac{f+\imath \tilde{f}}{2}$ and $\overline{w}=\frac{f-\imath \tilde{f}}{2}$ and integrating over the unit circle:
$$\bigg\|\bigg(\bigg|\frac{f+\imath \tilde{f}}{2}\bigg|^s+\bigg|\frac{f-\imath \tilde{f}}{2}\bigg|^s\bigg)^{\frac{1}{s}}\bigg\|_{L^p(\mathbb{T})}^p+2^{\frac{p}{s}-p}D_{p,s}\big|f(0)\big|^p\big|\cos\frac{p\pi}{2}\big|  \leq  A_{p,s}^p \|f\|_{L^p(\mathbb{T})}^p, $$
for $p\geq 2$ and $s>0.$ Similarly, for $1<p\leq \frac{4}{3}$ and $s\leq \sec^2\frac{\pi}{2p}$
we get:
$$\bigg\|\bigg(\bigg|\frac{f+\imath \tilde{f}}{2}\bigg|^s+\bigg|\frac{f-\imath \tilde{f}}{2}\bigg|^s\bigg)^{\frac{1}{s}}\bigg\|_{L^p(\mathbb{T})}^p+2^{\frac{p}{s}-p}D_{p,s}\big|f(0)\big|^p  \leq  A_{p,s}^p \|f\|_{L^p(\mathbb{T})}^p. $$

To see that $A_{p,s}$ is sharp, we use the same family of functions $f_{\gamma}.$
Indeed, since $\frac{f_{\gamma}+\imath \tilde{f}_{\gamma}}{2}=\frac{\alpha+\beta}{2}g_{\gamma}$ and $\frac{f_{\gamma}-\imath \tilde{f}_{\gamma}}{2}=\frac{\alpha-\beta}{2}\overline{g_{\gamma}},$ from $|\Im g_{\gamma}|=\tan\frac{\pi\gamma}{2}\Re g_{\gamma}$ we obtain:
$$ \frac{\big\|\big(\big|\frac{f_{\gamma}+\imath \tilde{f}_{\gamma}}{2}\big|^s+\big|\frac{f_{\gamma}-\imath \tilde{f}_{\gamma}}{2}\big|^s\big)^{\frac{1}{s}}\big\|_{L^p(\mathbb{T})}}{\|f_{\gamma}\|_{L^p(\mathbb{T})}}= \frac{\|\big(|\frac{\alpha+\beta}{2}g_{\gamma}|^s+|\frac{\alpha-\beta}{2}\overline{g_{\gamma}}|^s\big)^{\frac{1}{s}}\|_{L^p({\mathbb{T}})}}{\|\frac{\alpha+\beta}{2}g_{\gamma}+\frac{\alpha-\beta}{2}\overline{g_{\gamma}}\|_{L^p({\mathbb{T}})}}.$$
Now the result follows by the arguments already used in the above considerations.
\end{proof}
Using the well-known trick of Zygmund (\cite{ZYGMUND}, Chapter XVI, Theorem 3.8.) and Corollary \ref{corollary1} in the same way as in the paper \cite{HV.JFA}, we obtain the real-line analog of Theorem \ref{HVcp}.
\begin{theorem}
\label{corollary2}
Let $f$ be an $L^p(\mathbb{R})$ function and $P_+f$ and $P_-f$ its analytic and co-analytic projections. Then the following sharp estimate holds:
\begin{equation}
\label{realni}
\|(|P_+f|^s+|P_-f|^s)^{\frac{1}{s}}\|_{L^p(\mathbb{R})} \leq A_{p,s} \|f\|_{L^p(\mathbb{R})}
\end{equation} 
with the same constant $A_{p,s}$ for the values of $p$ and $s$ as in Theorem \ref{HVcp}.	 
\end{theorem}	
Let $H_{\pm}^{j}:=\{x=(x_1,x_2,\dots,x_n)\in \mathbb{R}^n: x_j\gtrless 0\},$ $j=1,2,\dots,n$ and $P_{\pm}^j f=\mathcal{F}^{-1}\big(\chi_{H_{\pm}^{j}}\mathcal{F} f\big)$ the corresponding half-space multipliers. This operators can be seen as analytic and co-analytic projections with respect to $j-$th variable. Then by the argument from (\cite{STEIN}, Chapter 4, Section 4.2.3.),
we infer the next corollary:
\begin{corollary}
	\label{corollary3}
	Let $f$ be an $L^p(\mathbb{R}^n)$ function and $P^j_+f$ and $P^j_-f$ be the half-space multipliers corresponding to half-spaces $H_{+}^j$ and $H_{-}^j,$ respectively. Then the following sharp estimate holds:
	\begin{equation}
	\label{mnozioci}
	\|(|P^j_+f|^s+|P^j_-f|^s)^{\frac{1}{s}}\|_{L^p(\mathbb{R}^n)} \leq A_{p,s} \|f\|_{L^p(\mathbb{R}^n)}
	\end{equation} 
	with the same constant $A_{p,s}$ for the values of $p$ and $s$ as in Theorem \ref{HVcp}.	 	
\end{corollary} 
Finally, we give a probabilistic generalization of Theorem \ref{HVcp} which can be proved by using the inequalities from the next section with the help of Theorem 2.1 from \cite{OSEKOWSKI} as it is done in Theorem 4.2. in the same paper (\cite{OSEKOWSKI}):  
\begin{theorem}
	\label{MARTINGALE}
	If $(W,Z)$ is an analytic martingale with values in $\mathbb{C}\times \mathbb{C}$, such that $Z_0=0$ or $W_0=0,$ then there holds the sharp inequality:
	\begin{equation}
	\label{eq:martineq}
	\|\big(|W|^s+|Z|^s\big)^{\frac{1}{s}}\|_p \leq A_{p,s} \|W+\overline{Z}\|_p,
	\end{equation} 
	for $p\geq 2,$ $s>0$ and $A_{p,s}$ as defined after (\ref{eq:opstanorma2}).
	The similar holds for $1<p\leq \frac{4}{3}$ and $s\leq\sec^2\frac{\pi}{2p}$ with the constant as in (\ref{eq:norm12}).
\end{theorem}
Inequalities with plurisubharmonic minorants that we prove here are very sharp and we use a  specific method of proving the appropriate "elementary" inequalities. In Section 3 we describe how proofs of Theorems 
\ref{HVcp} and \ref{MARTINGALE} reduce to these inequalities. Section 4 contains the main idea of their proof for $p\geq 2$ and $s\leq \csc^2\frac{\pi}{2p}$. As the reader can see it is based on the existence and differentiability of a certain implicit function and an inequality for hyperbolic functions. These most important steps in our proof are given in Sections 5 and 6, respectively. In Section 7 we give a proof of Theorem \ref{HVcp} for $p\geq 2$ and $s>\csc^2\frac{\pi}{2p}$, while in Section 8 we obtain the proof of Theorem \ref{HVcp} for $1<p\leq \frac{4}{3}.$ Finally, Section 9 is reserved for a remark on an isoperimetric inequality for harmonic functions.

\section{\textbf{Theorem \ref{HVcp} via inequalities with plurisubharmonic minorants}}

In this section we will state the appropriate "elementary" inequalities with plurisubharmonic minorants used in proof of our results. It can be seen that they are equal to that used in \cite{HV.OTAA} and \cite{HV.JFA}, up to a multiplicative constant. This problem of finding a good minorant is an example a plurisubharmonic obstacle problem and an interesting introduction to this topic is given in (\cite{VASYUNINVOLBERG}, pages 195-205). More precisely, we use the following lemma:
\begin{lemma}
	\label{lemma1}
	For $p \geq 2$ and $s>0$ there holds the inequality:
	\begin{equation}
		 \label{eq:opsta2b}
		\left(\frac{|z|^s + |w|^s}{2}\right)^{\frac{p}{s}}
		\leq C_{p,s}|z+\overline {w}|^p -(rR)^{\frac{p}{2}}
		D_{p,s}v_p \bigg(\frac{u+t}{2}\bigg),
		\end{equation}
		where $u = \arg z,$ $t= \arg w,$ $r=|z|, R=|w|$ and
		\begin{equation*} v_p (t)  =
		\begin{cases}
		-\cos(p(\frac {\pi}{2}- |t|)), & \mbox{if}\  \frac {\pi}{2}-\frac {\pi}{p}<|t|\le \frac {\pi}{2}; \\
		\max\{|\cos(p(\frac {\pi}{2}- t))|, |\cos(p(\frac{\pi}{2}+t))|\}, & \mbox{if}\ |t|\le \frac{\pi}{2}-\frac{\pi}{p}.
		\end{cases}
		\end{equation*}
		and  $v_p(t) = v_p (\pi - |t|)$ for $\frac \pi 2\le |t|\le \pi$. 
		Constants $C_{p,s}$ and $D_{p,s}$ are defined in Section 7.
		
		Specially, for $s\leq\csc^2\frac{\pi}{2p}$ we have:
		\begin{equation}
		\label{eq:main2b}
		\left(\frac{|z|^s + |w|^s}2\right)^{\frac{p}{s}}
		\leq \frac{|z+\overline {w}|^p }{2^p \sin^p\frac{\pi}{2p}} - (rR)^{\frac{p}{2}}
		\cot \frac{\pi}{2p}v_p \bigg(\frac{u+t}{2}\bigg),
		\end{equation}
For   $s\leq\sec^2\frac{\pi}{2p}$ and  $1<p<\frac{4}{3}$ there holds
\begin{equation}
\label{eq:main12}
\left(\frac{|z|^s+|w|^s}{2}\right)^{\frac{p}{s}}\leq \frac{|z+\overline{w}|^p}{2^p\cos^p \frac{\pi}{2p} } -  \tan\frac{\pi}{2p}\Re (zw)^{\frac  {p}{2}},\quad z\in \mathbf{C},\, w\in \mathbf{C}.
\end{equation}
\end{lemma}
We will define $C_{p,s}$ and $D_{p,s}$ later. In this section, we will just use that they are both positive. 

It is well-known that if $U(re^{\imath t})=r^pv_p(t)$ is subharmonic in $\mathbb{D}$, then $E(z,w):=(rR)^{\frac{p}{2}}v_p(\frac{u+t}{2})$ where $u = \arg z,$ $t= \arg w,$ $r=|z|, R=|w|$ is plurisubharmonic function (this can be found in \cite{RANGE} or \cite{HORMANDER}). The proof of the subharmonicity of $U$ can be found in  \cite{VERBITSKY.ISSLED}, where this function was originally defined. The minorant from (\ref{eq:main12}) was also used by Hollenbeck-Verbitsky and Kalaj in \cite{HV.JFA} and \cite{KALAJ.TAMS}, respectively. Further, putting $z=P_+f(\zeta)$ and $w=\overline{P_-f(\zeta)}$ and integrating these inequalities, together with using
$$ \frac{1}{2\pi}\int_{\mathbf{T}}E(P_+f(\zeta),P_-f(\zeta))|d\zeta|\geq 0,$$
we infer the sharp norm estimates from Theorem \ref{HVcp} with 
\begin{equation}
\label{Aps}
A_{p,s}=2^{\frac{1}{s}}C_{p,s}^{\frac{1}{p}}.
\end{equation} In case of  $s\leq\max\{\csc^2\frac{\pi}{2p},\sec^2\frac{\pi}{2p}\}$ it takes the form given by (\ref{eq:C12}) and (\ref{eq:C2b}).

We will describe some reductions of these inequalities for the sake of completeness even it was done in \cite{HV.OTAA}. First, these inequalities are invariant under the transformation $(z,w)\rightarrowtail(ze^{\imath \varphi}, we^{-\imath \varphi})$, so we can set $w=1$ and rewrite it as:
\begin{equation}
\label{eq:FR2b}
-\bigg(\frac{1+r^s}{2}\bigg)^{\frac{p}{s}}+C_{p,s}\big(1+r^2+2r\cos t\big)^{\frac{p}{2}}-D_{p,s}r^{\frac{p}{2}}v_p\bigg(\frac{t}{2}\bigg) \geq 0,
\end{equation} 
for $p\geq 2$ and
\begin{equation}
\label{eq:FR12}
-\bigg(\frac{1+r^s}{2}\bigg)^{\frac{p}{s}}+ \frac{(1+r^2+2r\cos t)^{\frac{p}{2}}}{2^p\cos^p\frac{\pi}{2p}}-r^{\frac{p}{2}}\tan\frac{\pi}{2p}\cos\frac{t p}{2}\geq 0,
\end{equation} 
for $1<p<2.$ Multiplying both of them with $r^{-p},$ we see that they are non-negative for $r\geq 1$ if and only if they are non-negative for $0<r\leq 1.$ Hence, it is enough to prove them for $0<r\leq 1.$ The function on the left-hand side in the inequality (\ref{eq:FR12}) is an even function on $t,$ so we can assume $t \in [0,\pi].$ 

Now, we reproduce the approach from (\cite{HV.OTAA}) to show that for the inequality (\ref{eq:FR2b}) we can assume $t \in [\pi-\frac{2\pi}{p},\pi].$ Indeed, since for $t \in [0,\pi-\frac{2\pi}{p}],$ we have $0\leq v_p(\frac{t}{2})\leq 1$ and $-\cos(\frac{p}{2}(\pi-t))$ decreases from $1$ to $-1$ on $[\pi-\frac{2\pi}{p},\pi]$, there exists $\theta \in [\pi-\frac{2\pi}{p},\pi]$ such that $v_p(\frac{t}{2})=-\cos\frac{p(\pi-\theta)}{2}=v_p(\frac{\theta}{2}).$ Therefore, since $\cos t\geq \cos\theta,$
we easily see
\begin{align*}
&-\bigg(\frac{1+r^s}{2}\bigg)^{\frac{p}{s}}+C_{p,s}\big(1+r^2+2r\cos t\big)^{\frac{p}{2}}-D_{p,s}r^{\frac{p}{2}}v_p\bigg(\frac{t}{2}\bigg)\geq\\
&-\bigg(\frac{1+r^s}{2}\bigg)^{\frac{p}{s}}+C_{p,s}\big(1+r^2+2r\cos \theta\big)^{\frac{p}{2}}-D_{p,s}r^{\frac{p}{2}}v_p\bigg(\frac{\theta}{2}\bigg)\geq 0.
\end{align*}
 Similar reduction, that continues to hold here, has been done in \cite{MELENTIJEVICMARKOVIC} with minorants firstly defined in \cite{KALAJ.TAMS}, in both cases $2\leq p\leq 4$ and $p>4.$ 

Finally, replacing $t$ with $\pi-t$ and dividing by $r^{\frac{p}{2}}$ and using a change of variable $r=e^{-y}$ we arrive at
\begin{equation}
\label{eq:NFR2b}
-\cosh^{\frac{p}{s}}\big(\frac{sy}{2}\big)+C_{p,s}\big(\cosh y-\cos t\big)^{\frac{p}{2}}+D_{p,s}\cos\frac{tp}{2}\geq 0,
\end{equation}
for $y\geq 0$, $t \in [0,\frac{2\pi}{p}]$ and $p\geq 2.$ Similarly, 
in case  $1<p<\frac{4}{3}$ and $s=\sec^2\frac{\pi}{2p},$ we have to prove
\begin{equation}
\label{eq:NFR12}
-\cosh^{\frac{p}{s}}\big(\frac{sy}{2}\big)+\frac{(\cosh y+\cos t)^{\frac{p}{2}}}{(1+\cos\frac{\pi}{p})^{\frac{p}{2}}}-\tan\frac{\pi}{2p}\cos\frac{tp}{2}\geq 0,
\end{equation}
for $y\geq 0$ and $0\leq t\leq \pi.$ 

We will now restate the results from Section 2 in \cite{MELENTIJEVICMARKOVIC} which we will be useful later in proving (\ref{eq:NFR2b}) and (\ref{eq:NFR12}). Also, we will give a shorter proof with new notation.
\begin{lemma}
	\label{lemaoKfunkciji}
	Let $s>0,$ $p \geq 2$ and let $C_{p,s}$ denotes the maximum of the function
	\begin{equation}
	\label{Kfunkcija}
	K(y):=\frac{\cosh^{\frac{p}{s}}\big(\frac{s y}{2}\big)}{\big(\cosh y-\cos\frac{\pi}{p}\big)^{\frac{p}{2}}}.
	\end{equation}
	Then:
	\begin{itemize}
	\item For $s\leq \csc^2\frac{\pi}{2p},$ the maximum is attained at $y=0$ and its value is \\ $C_{p,s}=\big(\frac{s}{2}\big)^{\frac{p}{2}};$
	\item For $s>\csc^2\frac{\pi}{2p},$ the maximum is attained at some $\tilde{y}>0.$ Specially, for $p=2$, the supremum is equal to the limit in $+\infty$ i.e. $1.$
	\end{itemize}
\end{lemma}
\begin{proof}
Let $p>2.$ From
$$K'(y)=\frac{p\cosh^{\frac{p}{s}-1}(\frac{s y}{2})\sinh(\frac{sy}{2})}{2\big(\cosh y-\cos\frac{\pi}{p}\big)^{\frac{p}{2}+1}}\bigg(\frac{\sinh\big(\frac{(s-2)y}{2}\big)}{\sinh\big(\frac{sy}{2}\big)}-\cos\frac{\pi}{p}\bigg)$$
 we see that it is enough to examine the sign of  $\frac{\sinh\big(\frac{(s-2)y}{2}\big)}{\sinh\big(\frac{sy}{2}\big)}-\cos\frac{\pi}{p}$. Since its derivative is equal to $\frac{a\cosh(ay)\sinh(by)-b\cosh(by)\sinh(ay)}{\sinh^2(by)}$ and $\big(a\cosh(ay)\sinh(by)-b\cosh(by)\sinh(ay)\big)'=(a^2-b^2)\sinh(ay)\sinh(by)\leq 0,$ where $a=\frac{s-2}{2}$ and $b=\frac{s}{2},$ we conclude that $\frac{\sinh\big(\frac{(s-2)y}{2}\big)}{\sinh\big(\frac{sy}{2}\big)}-\cos\frac{\pi}{p}$ decreases in $y$. If $s\leq\csc^2\frac{\pi}{2p},$ then $\frac{\sinh\big(\frac{(s-2)y}{2}\big)}{\sinh\big(\frac{sy}{2}\big)}-\cos\frac{\pi}{p}$ has value in $0$ is $\leq 0$, hence $K(y)$ is decreasing and attains its maximum in $0$ equal $C_{p,s}=\big(\frac{s}{2}\big)^{\frac{p}{2}}.$ For $s>\csc^2\frac{\pi}{2p}$ considered expression is positive in $0$ and tends to $-\cos\frac{\pi}{p}<0$ when $y \rightarrow +\infty.$ Therefore, it has unique zero $\tilde{y}$ such that $K'(y)>0$ for $(0,\tilde{y})$ and $K'(y)<0$ for $(\tilde{y},+\infty),$ i.e. $\max_{y \in \mathbb{R}}K(y)=K(\tilde{y})=C_{p,s}.$ We leave the case of $p=2$ to the reader as an easy exercise. 
 \end{proof}
We can proceed similarly in the case $1<p<2,$ see \cite{MELENTIJEVICMARKOVIC}.
In the upcoming sections we will prove Lemma 1 in the range specified above.

\section{\textbf{Proof of the Lemma 1 for $s=\csc^2\frac{\pi}{2p}$ and $p\geq 2$}}

Now, we concentrate on the proof of the inequality (\ref{eq:NFR2b}) for $s=\csc^2\frac{\pi}{2p}$. In this section we describe the method used in its proof, while some steps will be postponed for the next two sections. 
Let us consider the following function:
\begin{equation}
\label{Phiprva}
\Phi(y,t)=-\cosh^{\frac{p}{s}}\big(\frac{sy}{2}\big)+\frac{(\cosh y-\cos t)^{\frac{p}{2}}}{(1-\cos\frac{\pi}{p})^{\frac{p}{2}}}+\cot\frac{\pi}{2p}\cos\frac{tp}{2}.
\end{equation}
Our aim is to prove its non-negativity for $y \geq 0$ and $t \in [0,\frac{2\pi}{p}].$

We easily find that
\begin{align*}
\frac{2}{p} \frac{\partial \Phi}{\partial t}(y,t)&= \frac{\big(\cosh y-\cos t\big)^{\frac{p}{2}-1}\sin t}{(1-\cos\frac{\pi}{p})^{\frac{p}{2}}}-\cot\frac{\pi}{2p}\sin\frac{tp}{2}\\
&\geq \sin t\bigg( \frac{\big(1 -\cos t\big)^{\frac{p}{2}-1}}{(1-\cos\frac{\pi}{p})^{\frac{p}{2}}}-\cot\frac{\pi}{2p}\frac{\sin\frac{tp}{2}}{\sin t}\bigg)\\
&\geq \sin t\bigg( \frac{\big(1 -\cos\frac{\pi}{p}\big)^{\frac{p}{2}-1}}{(1-\cos\frac{\pi}{p})^{\frac{p}{2}}}-\frac{\cot\frac{\pi}{2p}}{\sin \frac{\pi}{p}}\bigg)=0,
\end{align*}
by the claim from the proof of Lemma \ref{odnossinusa}. Therefore, $\Phi(y,t)$ increases in $t$ for $t \in [\frac{\pi}{p},\frac{2\pi}{p}]$ and $\Phi(y,t)\geq \Phi(y,\frac{\pi}{p})\geq 0.$ We are now reduced to $t \in [0,\frac{\pi}{p}].$
Validity of (\ref{eq:NFR2b}) for $y=0$ follows from the known results (\cite{KALAJ.TAMS},\cite{MELENTIJEVICMARKOVIC}).
From 
$$\frac{2}{p} \frac{\partial \Phi}{\partial y}(y,t)= -\cosh^{\frac{p}{s}-1}\big(\frac{sy}{2}\big)\sinh\big(\frac{sy}{2}\big)+\sinh y\frac{(\cosh y-\cos t)^{\frac{p}{2}-1}}{(1-\cos\frac{\pi}{p})^{\frac{p}{2}}}$$
for $y>0$ we will prove that, for every fixed $t \in [0,\frac{\pi}{p})$ there exists a $y_p(t) \in \mathbb{R}$ such that $\frac{\partial \Phi}{\partial y}<0$ for $y \in (0,y_p(t)),$ while we have $\frac{\partial \Phi}{\partial y}>0$ for $y \in (y_p(t), +\infty ).$

The function $y_p(t)$ is well defined, differentiable and monotone (Section 6) on its domain and we can consider the function $\Phi(y_p(t),t)$. 
The crux of the proof is the fact that $\frac{d}{dt}\Phi(y_p(t),t)\leq 0.$ Hence, our function is decreasing and $\Phi(y_p(t),t)\geq\Phi(y_p(\frac{\pi}{p}),\frac{\pi}{p})=\Phi(0,\frac{\pi}{p})=0$.
We have
$$\frac{2}{p}\frac{d}{dt}\Phi(y_p(t),t)=$$
$$y'_p(t)\bigg[-\cosh^{\frac{p}{s}-1}\bigg(\frac{sy_p(t)}{2}\bigg)\sinh\bigg(\frac{sy_p(t)}{2}\bigg)+\sinh y_p(t)\frac{(\cosh y_p(t)-\cos t)^{\frac{p}{2}-1}}{(1-\cos\frac{\pi}{p})^{\frac{p}{2}}}\bigg]+$$
$$\frac{(\cosh y_p(t)-\cos t)^{\frac{p}{2}-1}\sin t}{(1-\cos\frac{\pi}{p})^{\frac{p}{2}}}  -\cot\frac{\pi}{2p}\sin\frac{tp}{2}=\frac{(\cosh y_p(t)-\cos t)^{\frac{p}{2}-1}\sin t}{(1-\cos\frac{\pi}{p})^{\frac{p}{2}}}  -\cot\frac{\pi}{2p}\sin\frac{tp}{2},$$
since the expression inside the brackets is equal to zero ($y_p(t)$ is the zero of $\frac{\partial \Phi}{\partial y}$ for fixed $t$).

Therefore, it is enough to prove that 
$$\frac{(\cosh y_p(t)-\cos t)^{\frac{p}{2}-1}}{(1-\cos\frac{\pi}{p})^{\frac{p}{2}}}  \leq \cot\frac{\pi}{2p}\frac{\sin\frac{tp}{2}}{\sin t}$$
for $t \in (0,\frac{\pi}{p}).$ The following lemma is the first step in its proof:
\begin{lemma}
\label{odnossinusa}	
 We have the following inequality:
	$$\frac{\sin\frac{tp}{2}}{\sin t}\geq E_p\cos^{\frac{p}{2}-1}t$$
for $0 \leq t \leq \frac{\pi}{p},$ $p \geq 2$ with $E_p=\frac{1}{\sin\frac{\pi}{p}\cos^{\frac{p}{2}-1}\frac{\pi}{p}}.$
\end{lemma} 
\begin{proof}
Let us consider the function $f(t)=\log\sin\frac{tp}{2}-\log\sin t-(\frac{p}{2}-1)\log\cos t.$ 
We find the derivative
 \begin{align*}
 f'(t)&=\frac{\frac{p}{2}\cos\frac{tp}{2}}{\sin\frac{tp}{2}}-\frac{\cos t}{\sin t}+\big(\frac{p}{2}-1\big)\frac{\sin t}{\cos t}\\
 &=\frac{1}{\sin\frac{tp}{2}\sin t\cos t}\bigg[\frac{p}{2}\cos\frac{tp}{2}\sin t\cos t-\cos^2 t\sin\frac{tp}{2}+\big(\frac{p}{2}-1\big)\sin^2 t\sin\frac{tp}{2}\bigg].
 \end{align*}
Note that the conclusion of our lemma follows from $f'(t)\leq 0,$ so we have to show that
$h(t)=\frac{p}{2}\cos\frac{tp}{2}\sin t\cos t-\cos^2 t\sin\frac{tp}{2}+\big(\frac{p}{2}-1\big)\sin^2 t\sin\frac{tp}{2}=\frac{p}{4}(\sin 2t\cos\frac{tp}{2}-\cos 2t\sin\frac{tp}{2})+(\frac{p}{4}-1)\sin\frac{tp}{2}=\frac{p}{4}\sin(\frac{p}{2}-2)t+(\frac{p}{4}-1)\sin\frac{tp}{2}\leq 0.$
This is equivalent to $\frac{\sin(\frac{p}{2}-2)t}{\sin\frac{tp}{2}}\geq \frac{p-4}{p}$ for $p>4$ and to $\frac{\sin(2-\frac{p}{2})t}{\sin\frac{tp}{2}}\geq \frac{4-p}{p}$ for $2\leq p \leq 4.$ Both of these estimates follows from the next statement:
\begin{claim}
If $0<\alpha<\beta$ and $t \in (0,\frac{\pi}{\beta}),$ then:
$$g(t)=\frac{\sin\alpha t}{\sin\beta t}$$
is monotone increasing function on $(0,\frac{\pi}{\beta}).$
\end{claim}

Indeed, $g'(t)=\frac{1}{\sin^2 \beta t}\big(\alpha \cos\alpha t\sin\beta t-\beta\cos\beta t\sin\alpha t\big)$ which is not smaller then zero, since the expression inside the brackets is $0$ at $0$ and its derivative on $t$ is $(\beta^2-\alpha^2)\sin\alpha t\sin\beta t\geq 0$.
\end{proof}
\textbf{Remark}. For $2\leq p\leq 4$, it can be proved that the exponent $\frac{p}{2}-1$ in the previous lemma can be replaced with not greater than $\cot^2\frac{\pi}{p},$ but this is still good and more convenient for our needs. 

Using Lemma \ref{odnossinusa} we are able to handle further with the inequality (\ref{eq:NFR2b}) working with the $y-$variable. It is sufficient to show the estimate:
$$ \tan\frac{\pi}{2p}\frac{(\cosh y_p(t)-\cos t)^{\frac{p}{2}-1}}{(1-\cos\frac{\pi}{p})^{\frac{p}{2}}} \leq
\frac{\cos^{\frac{p}{2}-1}t}{\sin\frac{\pi}{p}\cos^{\frac{p}{2}-1}\frac{\pi}{p}}.$$
Now, the left-hand side of the last inequality is equal to $\tan\frac{\pi}{2p}\frac{\cosh^{\frac{p}{s}-1}(\frac{sy}{2})\sinh (\frac{sy}{2})}{\sinh y}$; the right-hand side is, from the condition $\frac{\partial \Phi}{\partial y}(y_p(t),t)=0$, equal to 
$$\frac{1}{\sin\frac{\pi}{p}\cos^{\frac{p}{2}-1}\frac{\pi}{p}}\bigg[\cosh y-\big(1-\cos\frac{\pi}{p}\big)\bigg(\big(1-\cos\frac{\pi}{p}\big)\frac{\cosh^{\frac{p}{s}-1}(\frac{sy}{2})\sinh (\frac{sy}{2})}{\sinh y}\bigg)^{\frac{2}{p-2}}\bigg]^{\frac{p}{2}-1}.$$

Hence, we want to show that 
\begin{align*}&\tan\frac{\pi}{2p}\frac{\cosh^{\frac{p}{s}-1}(\frac{sy}{2})\sinh (\frac{sy}{2})}{\sinh y}\leq \\
&\frac{1}{\sin\frac{\pi}{p}\cos^{\frac{p}{2}-1}\frac{\pi}{p}}\bigg[\cosh y-\big(1-\cos\frac{\pi}{p}\big)\bigg(\big(1-\cos\frac{\pi}{p}\big)\frac{\cosh^{\frac{p}{s}-1}(\frac{sy}{2})\sinh (\frac{sy}{2})}{\sinh y}\bigg)^{\frac{2}{p-2}}\bigg]^{\frac{p}{2}-1},
\end{align*}
which, after multiplying by ${\sin\frac{\pi}{p}\cos^{\frac{p}{2}-1}\frac{\pi}{p}},$ taking the $\frac{2}{p-2}$th power of both sides and using $\tan\frac{\pi}{2p}\sin\frac{\pi}{p}\cos^{\frac{p}{2}-1}\frac{\pi}{p}=\big(2\sin^2\frac{\pi}{2p}\big)\big(1-2\sin^2\frac{\pi}{2p}\big)^{\frac{p}{2}-1}=\big(\frac{2}{s}\big)\big(1-\frac{2}{s}\big)^{\frac{p}{2}-1}$ and $1-\cos\frac{\pi}{p}=\frac{2}{s}$ reduces to the inequality 
$$ f(y)^{\frac{2}{p-2}} \leq \cosh y,$$
for $$f(y)=\frac{2}{s}\frac{\cosh^{\frac{p}{s}-1}(\frac{sy}{2})\sinh (\frac{sy}{2})}{\sinh y}.$$
\textbf{Remark}. Note that, after realizing that $y_p(t)$ is well defined (which is easy in this case), 
for $s=2$ the last inequality becomes the identity. Hence, we get a very quick proof of the inequality used in \cite{KALAJ.TAMS} to prove the result when $s=2.$

\section{\textbf{Main estimate for hyperbolic functions}}

Both existence of $y_p(t)$ and the conclusion on the monotonicity of the function $\Phi(y_p(t),t)$ on $(0,\frac{\pi}{p})$ depends on the inequality 
\begin{equation}
\label{hiperbolickeosnovna}
\frac{2}{s}\frac{\cosh^{\frac{p}{s}-1}(\frac{sy}{2})\sinh (\frac{sy}{2})}{\sinh y}\leq \cosh^{\frac{p}{2}-1}y
\end{equation}
for some range of $y.$

The next lemma will give a crucial step in estimating range of $y'$s for $p\geq 4.$ We will work with all $s \geq \csc^2\frac{\pi}{2p},$ hence the Lemma will be useful also in general setting.  
\begin{lemma} 
	\label{ocenaipsilona4b}
	For $p\geq 4$ and $s\geq \csc^2\frac{\pi}{2p}$ there holds the estimate:
	$$\cosh y- \bigg(\frac{\cosh^{\frac{p}{s}-1}(\frac{sy}{2})\sinh (\frac{sy}{2})}{C_{p,s}\sinh y}\bigg)^{\frac{2}{p-2}}>1$$ 
	for $y \geq \cosh^{-1}\big(\frac{1}{\cos\frac{\pi}{p}}\big).$\\
	
	Specially, for this range of $y$ and $s=\csc^2\frac{\pi}{2p}$ this reads as
	$$\cosh y-\frac{2}{s}\bigg(\frac{2}{s}\frac{\cosh^{\frac{p}{s}-1}(\frac{sy}{2})\sinh (\frac{sy}{2})}{\sinh y}\bigg)^{\frac{2}{p-2}}>1.$$ 
	\end{lemma}

\begin{proof}Since $\frac{\cosh^{\frac{p}{s}-1}(\frac{sy}{2})\sinh (\frac{sy}{2})}{\sinh y}<\frac{\cosh^{\frac{p}{s}}(\frac{sy}{2})}{\sinh y},$ it is enough to prove
	$$ \cosh^{\frac{p}{s}}\big(\frac{sy}{2}\big)<C_{p,s}\sinh y\big(\cosh y-1\big)^{\frac{p}{2}-1}.$$
	If we rewrite it as
	$$\bigg(\frac{e^{\frac{sy}{2}}+e^{-\frac{sy}{2}}}{2}\bigg)^{\frac{p}{s}}<C_{p,s} \bigg(\frac{e^y-e^{-y}}{2}\bigg) \bigg(\frac{e^y+e^{-y}}{2}-1\bigg)^{\frac{p}{2}-1},$$
	after dividing both sides by $e^{\frac{py}{2}},$ we get 
	\begin{equation}
	\label{ocenaresenja}
	\bigg(\frac{1+e^{-sy}}{2}\bigg)^{\frac{p}{s}}<C_{p,s}2^{-\frac{p}{2}}  \big(1-e^{-y}\big)^{p-1}\big(1+e^{-y}\big).
	\end{equation}
	Note that the left-hand side is monotone decreasing, while the right-hand side is increasing function, so the inequality follows if we prove it for $y=\cosh^{-1}\big(\frac{1}{\cos\frac{\pi}{p}}\big).$ Since
	$e^{-y}=\frac{1-\sin\frac{\pi}{p}}{\cos\frac{\pi}{p}}$, we easily calculate that
	$\big(1-e^{-y}\big)^{p-1}\big(1+e^{-y}\big)=\cot\frac{\pi}{2p}\frac{2^p\sin^p\frac{\pi}{2p}}{(1+\sin\frac{\pi}{p})^{\frac{p}{2}}}.$ In order to estimate the left-hand side of (\ref{ocenaresenja}), we will prove:
	\begin{equation}
	\label{eq:coscosh}
	\phi(x)=\cosh x \cos x\leq 1,\quad \text{where} \quad x=\frac{\pi}{p}<1.
	\end{equation}
	Indeed, by the power series expansion, we have:
	\begin{align*}
	\phi(x)&=\bigg(1+\frac{x^4}{4!}+\frac{x^8}{8!}+\dots\bigg)^2-\bigg(\frac{x^2}{2!}+\frac{x^6}{6!}+\frac{x^{10}}{10!}+\dots\bigg)^2\\
	&\leq\bigg(1+\frac{x^4\psi(x)}{12}\bigg)^2-x^4\psi^2(x)=1+\frac{x^4\psi(x)}{6}-x^4\psi^2(x)+\frac{x^8\psi^2(x)}{144},
	\end{align*}
	where $\psi(x)=\frac{1}{2!}+\frac{x^4}{6!}+\frac{x^8}{10!}+\dots$ Since $\psi(0)=\frac{1}{2}\leq\psi(x)\leq \psi(1)\leq 1$, we conclude
	$$\phi(x)\leq 1+\frac{x^4}{6}+\frac{x^8}{144}-\frac{x^4}{4}=1-\frac{x^4}{12}+\frac{x^8}{144}\leq 1.$$
	
	By convexity of the power-mean function and inequality (\ref{eq:coscosh}), we estimate the left-hand side as
	$$\bigg(\frac{1+e^{-s y}}{2}\bigg)^{\frac{p}{s}}\leq 2^{-\frac{p}{s}}\big(1+e^{-p y}\big)\leq 2^{-\frac{p}{s}}\big(1+e^{-\pi}\big).$$
	
	We estimate $C_{p,s}$ in the following way:
	$$C_{p,s}^{\frac{2}{p}}=\max_{y\geq 0}\frac{\cosh^{\frac{2}{s}}\big(\frac{s y}{2}\big)}{\cosh y-\cos\frac{\pi}{p}}\geq \frac{\cosh^{\frac{2}{s}}\big(\frac{s \hat{y}}{2}\big)}{\cosh \hat{y}-\cos\frac{\pi}{p}},$$
	where $\hat{y}=\log\frac{1}{\cos\frac{\pi}{p}}.$ Straightforward calculation gives:
	$$C_{p,s} 2^{-\frac{p}{2}}\geq \frac{1}{\sin^p\frac{\pi}{p}}\bigg(\frac{1+\cos^s\frac{\pi}{p}}{2}\bigg)^{\frac{p}{s}}> \frac{2^{-\frac{p}{s}}}{\sin^p\frac{\pi}{p}}.$$
	Hence, it is enough to prove:
	$$ \frac{\cot\frac{\pi}{2p}}{\cos^p\frac{\pi}{2p}(1+\sin\frac{\pi}{p})^{\frac{p}{2}}}\geq 1+e^{-\pi}.$$
	We introduce the change of variable $x=\frac{\pi}{2p}\leq\frac{\pi}{8}$ and consider the function 
	$$F(x):=\log\cot x-\frac{\pi}{2x}\log\cos x-\frac{\pi}{4x}\log\big(1+\sin(2x)\big).$$ We will prove that it is decreasing, hence $f(p):=\frac{\cot\frac{\pi}{2p}}{\cos^p\frac{\pi}{2p}(1+\sin\frac{\pi}{p})^{\frac{p}{2}}}$ increases in $p$ and the conclusion follows from $f(p)\geq f(4)=80\sqrt{2}-112>\frac{9}{8}>1+e^{-\pi}.$
	From
	$$F'(x)=\frac{\pi}{2x^2}\log\cos x+\frac{\pi}{4x^2}\log\big(1+\sin(2x)\big)+\frac{\pi \tan x}{2x}-\frac{\pi\cos(2x)}{2x(1+\sin(2x))}-\frac{2}{\sin(2x)},$$
	we find 
	$$F'(x)\leq \frac{\pi}{2x^2}\log\cos x+\frac{\pi}{4x^2}\log\big(1+2x\big)+\frac{\pi \tan x}{2x}-\frac{\pi\cos(2x)}{2x+4x^2}-\frac{1}{x}.$$
	Now, using estimate $\cos(2x)\geq 1-2x^2,$ we get
	$$\frac{4x^2}{\pi}F'(x)\leq \log(1+2x)+\log(\cos^2 x)+2x\tan x-\frac{2x(1-2x^2)}{1+2x}-\frac{4x}{\pi}=:\beta(x).$$
	Since $\beta''(x)=\frac{4(8x^3+12x^2+4x+1)^2}{(1+2x)^3}+\frac{2(2x\tan x+1)}{\cos^2 x}>0,$ we have 
	$\beta(x)\leq \max\{\beta(0),\beta(\frac{\pi}{8})\}\\=0$ and $F'(x)\leq 0.$
	\end{proof}

This exactly means that we have to work just with $y\leq \cosh^{-1}\big(\frac{1}{\cos\frac{\pi}{p}}\big)$, since for bigger values of $y$ the equality $\varphi(y)=\cos t$ is not possible. (Or, in other words, such $y'$s are not in the range of $y_p(t)$ for any $t!$)

Now, we will come back to our inequality (\ref{hiperbolickeosnovna}) and consider the function
$$G(y)=\frac{\cosh^{\frac{p}{s}-1}(\frac{sy}{2})\sinh (\frac{sy}{2})}{\sinh y\cosh^{\frac{p}{2}-1}y}.$$
In the rest of this section, we prove that it is $\leq \frac{s}{2}$ for the desired range of $y.$ More precisely, we need to prove it for $y \in [0,\cosh^{-1}\big(\frac{1}{\cos\frac{\pi}{p}}\big)]$ when $p\geq 4,$ while for $2\leq p\leq 4,$ we are able to prove it for all $y\geq 0.$ 
Its derivative is
$$G'(y)= \frac{\cosh^{\frac{p}{s}-2}(\frac{sy}{2})\cosh^{\frac{p}{2}-2}y}{2\sinh^2 y\cosh^{p-2}y}\bigg[\sinh(2y)\bigg(\frac{p}{2}\sinh^2(\frac{sy}{2})+\frac{s}{2}\bigg)-\sinh(sy)\bigg(\frac{p}{2}\sinh^2 y+1\bigg)\bigg].$$
We see that
$$G'(y)<0 \iff \sinh(2y)\bigg(\frac{p}{2}\frac{\cosh(sy)-1}{2}+\frac{s}{2}\bigg)<\sinh(sy)\bigg(\frac{p}{2}\frac{\cosh(2y)-1}{2}+1\bigg),$$
or, expanding and using transformation formulae for hyperbolic functions:
$$G'(y)<0 \iff g(y):=\frac{p}{4}\sinh((s-2)y)+\big(1-\frac{p}{4}\big)\sinh(sy)+\big(\frac{p}{4}-\frac{s}{2}\big)\sinh(2y)>0.$$

Hence, we can consider the sign of the function $g(y)$ and search for the monotonicity of $G(y).$ Using Taylor series expansion, we get:
$$g(y)=\sum_{k=1}^{+\infty}\bigg[\frac{p}{4}(s-2)^{2k+1}+(1-\frac{p}{4})s^{2k+1}+\big(\frac{p}{4}-\frac{s}{2}\big)2^{2k+1}\bigg]\frac{y^{2k+1}}{(2k+1)!}$$
Now, our aim is to investigate the sign of the coefficients i.e. the sign of the expression (denote it by $a_k$) inside the angular brackets. 

For $2\leq p\leq 3,$ we have $2\leq s\leq 4$ and
\begin{align*}
a_k\cdot2^{-2k-1}&=\frac{p}{4}\bigg(\frac{s-2}{2}\bigg)^{2k+1}+\big(1-\frac{p}{4}\big)\bigg(\frac{s}{2}\bigg)^{2k+1}+\big(\frac{p}{4}-\frac{s}{2}\big)\\
&= \big(1-\frac{p}{4}\big) \bigg(\bigg(\frac{s}{2}\bigg)^{2k+1}-1-\bigg(\frac{s}{2}-1\bigg)^{2k+1}\bigg)                      +1-\frac{s}{2}+ \bigg(\frac{s}{2}-1\bigg)^{2k+1}\\
&\geq \big(1-\frac{s}{4}\big) \bigg(\bigg(\frac{s}{2}\bigg)^{2k+1}-1-\bigg(\frac{s}{2}-1\bigg)^{2k+1}\bigg)                      +1-\frac{s}{2}+ \bigg(\frac{s}{2}-1\bigg)^{2k+1}\\
 &=\frac{s}{4} \bigg(\frac{s}{2}-1\bigg)^{2k+1}+ \big(1-\frac{s}{4}\big) \bigg(\frac{s}{2}\bigg)^{2k+1}-\frac{s}{4},
\end{align*}
since $\big(\big(\frac{s}{2}\big)^{2k+1}-1-\big(\frac{s}{2}-1\big)^{2k+1}\big)\geq 0$ from the elementary inequality $(a+b)^k\geq a^k+b^k,$ with $a,b\geq 0, k\geq 1$ and $s\geq p$     ($s-p=\frac{1}{\sin^2(\frac{\pi}{2p})}-p$ has the value $0$ at $p=2$ and $(s-p)'=\frac{\pi}{p^2}\frac{\cot\frac{\pi}{2p}}{\sin^2\frac{\pi}{2p}}-1>0,$ because $\frac{\pi^2}{4p^2}\geq \sin^2\frac{\pi}{2p}$ and $\frac{4}{\pi}\cot\frac{\pi}{2p}\geq 1,$ for $p\geq 2.$)

Now, rewriting the last bound and estimating using Jensen's inequality, we obtain:
\begin{align*}
&\frac{s}{4} \bigg(\frac{s}{2}-1\bigg)\bigg(\bigg(\frac{s}{2}-1\bigg)^2\bigg)^k+ \bigg(1-\frac{s}{4}\bigg)\frac{s}{2} \bigg(\bigg(\frac{s}{2}\bigg)^2\bigg)^k-\frac{s}{4}\\
&\geq\frac{s}{4}\bigg[\frac{(\frac{s^2}{8}-\frac{s}{4})(\frac{s}{2}-1)^2+(\frac{s}{2}-\frac{s^2}{8})(\frac{s}{2})^2}{\frac{s}{4}}\bigg]^k-\frac{s}{4}\\                              &=\frac{s}{4}\bigg(3\cdot\frac{s}{2}-1-\bigg(\frac{s}{2}\bigg)^2\bigg)^k-\frac{s}{4}\geq 0,
\end{align*}
since $3\cdot\frac{s}{2}-1-(\frac{s}{2})^2\geq 1, $ for $s \in [2,4].$ Hence, $a_k\geq 0$ for all $k\geq 1,$ $g(y)\geq 0$ and $G(y)$ is decreasing, which gives $G(y)\leq \lim_{y \rightarrow 0}G(y)=\frac{s}{2}$ for all $y>0.$

For $3\leq p\leq 4,$ we have that $s\geq 4$ and writing the $a_k$ as
$$a_k2^{-2k-1}=\bigg(\frac{s}{2}-1\bigg)^{2k+1}-\frac{s}{2}+1+\bigg(1-\frac{p}{4}\bigg)\bigg(\bigg(\frac{s}{2}\bigg)^{2k+1}-1-\bigg(\frac{s}{2}-1\bigg)^{2k+1}\bigg)$$
we easily that $\big(\frac{s}{2}-1\big)^{2k+1}\geq\frac{s}{2}-1$ for $s\geq 4$ and $k \geq 1$ and that both factors in the last summand are nonnegative because $p\leq 4$	and $(a+b)^k\geq a^k+b^k,$ with $a,b\geq 0, k\geq 1.$ The conclusion is the same as in the previous case.
Note that from those conclusions we easily infer that $G(y)$ is decreasing on $(0,+\infty)$ for $2\leq p\leq 4$ and the inequality (\ref{hiperbolickeosnovna}) is proved in this case.

For $p>4$ we rewrite $$a_k=s^{2k+1}\bigg[\frac{p}{4}\bigg(\frac{s-2}{s}\bigg)^{2k+1}+1-\frac{p}{4}+\big(\frac{p}{4}-\frac{s}{2}\big)\bigg(\frac{2}{s}\bigg)^{2k+1}\bigg]$$
and consider
$$B(t)=\frac{p}{4}\bigg(\frac{s-2}{s}\bigg)^{2t+1}+1-\frac{p}{4}+\big(\frac{p}{4}-\frac{s}{2}\big)\bigg(\frac{2}{s}\bigg)^{2t+1}$$
for $t\geq 1.$

We easily find that 
$$\tilde{B}(t):=\frac{1}{2}B'(t)\bigg(\frac{s}{2}\bigg)^{2t+1}=\frac{p}{4}\bigg(\frac{s-2}{2}\bigg)^{2t+1}\log(1-\frac{s}{2})+\big(\frac{p}{4}-\frac{s}{2}\big)\log\frac{2}{s}$$
and 
$$\tilde{B}'(t)=\frac{p}{4}\bigg(\frac{s-2}{2}\bigg)^{2t+1}\log\big(1-\frac{2}{s}\big)\log\big(\frac{s}{2}-1\big)^2<0,$$
since $s\geq \frac{1}{\sin^2{\frac{\pi}{8}}}>6,$ hence $\tilde{B}(t)$ is decreasing and 
$\lim_{t \rightarrow +\infty}\tilde{B}(t)=-\infty$ implies that $\tilde{B}$ is negative on $(0,+\infty)$ or positive on $(1,t_p)$ and negative on $(t_p,+\infty),$ for some $t_p>1.$ This means that $B(t)$ is strictly decreasing on $(0,+\infty)$ or strictly increasing on $(1,t_p)$ and strictly decreasing on $(t_p,+\infty).$ 

We see that 
$$B(1)=\frac{p}{4}\bigg(\frac{s-2}{s}\bigg)^3+1-\frac{p}{4}+\big(\frac{p}{4}-\frac{s}{2}\big)\bigg(\frac{2}{s}\bigg)^3=\frac{1}{s^2}\bigg(s\big(s-\frac{3p}{2}\big)+3p-4\bigg)>0,$$
which together with $\lim_{t \rightarrow +\infty}B(t)=1-\frac{p}{4}<0$
and the above monotonicity properties of $B(t)$ gives the existence of some $l=l(p)$ such that $B(t)>0$ for $t<l(p)$ and $B(t)<0$ for $t>l(p)$ and consequently $a_1, a_2,\dots,a_{k_0}>0$, while $a_{k_0+1},a_{k_0+2},\dots<0$ for some $k_0=k_0(p).$ (We use that $\frac{s}{p}$ is increasing in $p$ and its value for $p=4$ is greater than $\frac{3}{2}$, since $s(4)>6.$)
Coming back to our power series expansion of $g$ we have 
$$ g(y)=\sum_{k=1}^{k_0}c_ky^{2k+1}+\sum_{k=k_0+1}^{+\infty }c_ky^{2k+1},$$
where $c_k=\frac{a_k}{(2k+1)!}$ and the coefficients in the first sum are non-negative, while in the second are negative. 

We will proceed with the following lemma which will be also used later.
\begin{lemma}
	\label{koeficijentimonotonost}
	Let $f(x)=\sum_{k=0}^{+\infty}a_kx^k$ be a series where, for some $k_0 \in \mathbb{N},$ $a_1,a_2,\cdots, a_{k_0}$ are positive, while $a_{k_0+1},a_{k_0+2},\cdots$ are negative numbers.
	Then there exists some $x_0$ such that $f(x)$ is increasing on $[0,x_0]$ and decreasing on $[x_0,+\infty).$ Specially, $f(x)$ attains its minimum on an interval $[0,a]$ in one of its endpoints. 
\end{lemma}
\begin{proof}
	From $f^{(k_0+1)}(x)=\sum_{k=k_0+1}^{+\infty}a_kx^k<0,$ we conclude that $f^{(k_0+1)}(x)$ is monotone decreasing function and, since $f^{(k_0)}(0)=a_{k_0}>0$ and $f^{(k_0)}(+\infty)=-\infty,$ we infer that $f^{(k_0-1)}$ is increasing till some point and then decreasing. If we continue with this argument, by using that $f^{(l)}(0)>0$ and  $f^{(l)}(+\infty)=-\infty$ for $l\leq k_0$, it is easily seen that the same holds for $f(x)$ and we are done. 
\end{proof}

Applying this lemma to $g(y)$ we see that it is positive and then eventually negative on $[0,y']$, where $y'$ is the largest solution 
of 
 $$\varphi(y):=\cosh y-\frac{2}{s}\bigg(\frac{2}{s}\frac{\cosh^{\frac{p}{s}-1}(\frac{sy}{2})\sinh (\frac{sy}{2})}{\sinh y}\bigg)^{\frac{2}{p-2}}=1.$$
Let us note first that $y', $ by the Lemma \ref{ocenaipsilona4b} is well defined. Also, for $y>y'$ it is not possible that $\varphi(y)\leq 1$, since continuity of $\varphi(y)$ implies the existence of some $y_0>y'$ with $\varphi(y_0)=1,$ contrary to the definition of the $y'.$ This excuses why it is enough to prove (\ref{hiperbolickeosnovna}) on $[0,y'].$

By the Lemma \ref{ocenaipsilona4b}, we have $y'<\cosh^{-1}\big(\frac{1}{\cos\frac{\pi}{p}}\big).$ Since $G(y)$ is monotone decreasing and then increasing, it means that for showing that $G(y)\leq\frac{s}{2}$ for $y \in [0,y']$ it is enough to prove it for $y=y'.$ (Note that $G(0):=\lim_{y \rightarrow 0}G(y)=\frac{s}{2}.$) For $y=y',$ from $\varphi(y')=1$ we get $1-(\frac{2}{s})^{\frac{p}{p-2}}G(y')^{\frac{2}{p-2}}=\frac{1}{\cosh y'}$, while $1-(\frac{2}{s})^{\frac{p}{p-2}}G(0)^{\frac{2}{p-2}}=\cos\frac{\pi}{p}.$ Thus, $G(y')\leq G(0)$ is equivalent to
$$\frac{1}{\cosh y'}\geq \cos\frac{\pi}{p},$$
which easily follows from already deduced $y'<\cosh^{-1}\big(\frac{1}{\cos\frac{\pi}{p}}\big).$\\
 \textbf{Remark}. We could work with $y$ from the whole interval $[0,\cosh^{-1}\big(\frac{1}{\cos\frac{\pi}{p}}\big)].$

\section{\textbf{The function $y_p(t)$ is well defined, differentiable and monotone decreasing}}

This section contains detailed explanation on the definition of the function $y_p(t)$ and its basic properties that we used here. In order to do this, we will 
consider the sign of the $\frac{\partial \Phi}{\partial y}$ for $y>0.$ It is the same as the sign of
$$\varphi_t(y):=\varphi(y)-\cos t=\cosh y-\cos t-\frac{2}{s}\bigg(\frac{2}{s}\frac{\cosh^{\frac{p}{s}-1}(\frac{sy}{2})\sinh (\frac{sy}{2})}{\sinh y}\bigg)^{\frac{2}{p-2}}.$$

By Lemma \ref{ocenaipsilona4b} for $p>4$ we always have $\varphi_t(y)>0$ for $y\geq \cosh^{-1}\big(\frac{1}{\cos\frac{\pi}{p}}\big),$ hence
$\frac{\partial \Phi}{\partial y}>0,$ $\Phi$ increases in $y$ and its minimum in $y$ for fixed $t$ is achieved for some $y_p(t) \in (0,\cosh^{-1}\big(\frac{1}{\cos\frac{\pi}{p}}\big)).$ But, we need to explore further the function
$$\varphi(y)=\cosh y-\frac{2}{s}f(y)^{\frac{2}{p-2}}$$
to conclude that such $y_p(t)$ exists and that it is unique. 

We proceed with the following lemma.
\begin{lemma}
	\label{monot}
	The function $\varphi(y)$ is monotone increasing for $y \in (0,\cosh^{-1}\big(\frac{1}{\cos\frac{\pi}{p}}\big))$ for $p>4$, while for $2\leq p \leq 4$ it is increasing on the whole $(0,+\infty).$	
\end{lemma}
\begin{proof}
From
$$\varphi'(y)=\sinh y-\frac{2}{s}\frac{2}{p-2}f(y)^{\frac{2}{p-2}-1}f'(y)$$
and 
$$f'(y)=\frac{\cosh^{\frac{p}{s}-2}\big(\frac{sy}{2}\big)}{s \sinh^2 y}\bigg[ p\sinh^2\big(\frac{sy}{2}\big)\sinh y+ s \sinh y-\cosh y\sinh\big(sy\big)\bigg],  $$
we have:
$$\varphi'(y)=\sinh y-\frac{4}{s(p-2)}\frac{f(y)^{\frac{2}{p-2}}}{\sinh y \sinh\big(sy\big) }\bigg[ p\sinh^2\big(\frac{sy}{2}\big)\sinh y+ s \sinh y-\cosh y\sinh\big(sy\big)\bigg].$$
It is obvious that if $p\sinh^2\big(\frac{sy}{2}\big)\sinh y+ s \sinh y-\cosh y\sinh\big(sy\big)\leq0,$ we have $\varphi'(y)\geq 0,$ while in the other case we use 
the estimate $f(y)^{\frac{2}{p-2}}\leq \cosh y$ to get:
$$\varphi'(y)=\sinh y-\frac{4}{s(p-2)}\frac{\cosh y}{\sinh y \sinh\big(sy\big) }\bigg[ p\sinh^2\big(\frac{sy}{2}\big)\sinh y+ s \sinh y-\cosh y\sinh\big(sy\big)\bigg].$$
We see that a proof of our lemma depends on the main estimate for hyperbolic functions and that is why we have the restriction for the domain in case $p>4.$

Therefore,  $\varphi'(y)\geq 0$ follows from
$$p \sinh(2y)\frac{\cosh(s y)-1}{2}+s \sinh(2y)-\sinh(sy)\cosh(2y)-\sinh(sy)$$
$$\leq \frac{s}{2}\big(\frac{p}{2}-1\big)\big(\cosh(2y)-1\big)\sinh(sy).$$
This is equivalent with 
$$-1+\frac{p}{2}\frac{\cosh y\sinh(\frac{s y}{2})}{\sinh y\cosh(\frac{s y}{2})}+\frac{s\sinh(2y)-2\sinh(s y)}{\sinh(s y)(\cosh(2y)-1)}\leq \frac{s}{2}\big(\frac{p}{2}-1\big).$$
Since $\frac{\cosh y\sinh(\frac{s y}{2})}{\sinh y\cosh(\frac{s y}{2})}$
is monotone decreasing (its derivative is equal to  $\\\frac{s\sinh(2y)-2\sinh(sy)}{4\sinh^2y\cosh^2(\frac{s y}{2})}\leq 0$), we conclude
\begin{equation}
\label{monotonost1}
\big(\frac{p}{2}-1\big)\frac{\cosh y\sinh(\frac{s y}{2})}{\sinh y\cosh(\frac{s y}{2})}\leq \frac{s}{2}\big(\frac{p}{2}-1\big).
\end{equation}
Therefore, it is enough to prove 
\begin{equation}
\label{monotonost2}
-1+\frac{\cosh y\sinh(\frac{s y}{2})}{\sinh y\cosh(\frac{s y}{2})}+\frac{s\sinh(2y)-2\sinh(s y)}{\sinh(s y)(\cosh(2y)-1)}\leq 0.
\end{equation}
After multiplying both sides by $4\sinh^2y\sinh(\frac{sy}{2})\cosh(\frac{sy}{2}),$ we get
\begin{align*}
&-\sinh(sy)-\sinh((s-2)y)+(s-1)\sinh(2y)\\
&=2\cosh y\bigg((s-1)\sinh(y)-\sinh((s-1)y)\bigg)\leq 0.
\end{align*}

In any case, we can conclude the main result of this section using Lemma \ref{ocenaipsilona4b} and Lemma \ref{monot}. For $2\leq p \leq 4,$ $\varphi(y)$ is increasing and consequently the same holds for $\varphi_t(y)$ with fixed $t \in (0,\frac{\pi}{p}),$ therefore $\lim_{y \rightarrow 0}\varphi_t(y)=\frac{s-2}{s}-\cos t=\cos\frac{\pi}{p}-\cos t<0$	and $\lim_{y \rightarrow +\infty}\varphi_t(y)=+\infty$ gives the existence of the unique zero $y_p(t)$  of the considered function. The Implicit function theorem gives also that $y_p(t)$ is differentiable. 
	
For $p>4,$ Lemma \ref{ocenaipsilona4b} gives $\varphi_t(y)=\varphi(y)-\cos t>1-\cos t>0,$ for $y \geq y'$ while for $y \in (0,y')),$ we have $\lim_{y \rightarrow 0}\varphi_t(y)=\frac{s-2}{s}-\cos t=\cos\frac{\pi}{p}-\cos t<0$ and $\varphi_t(y')=\varphi(y')-\cos t>1-\cos t>0$ which imply, by the Lemma \ref{ocenaipsilona4b}, that there exists the unique zero $y_p(t) \in (0,y')$ of $\varphi_t(y)$ for each $t \in (0,\frac{\pi}{p}).$ Again, by the Implicit function theorem, we conclude the differentiability of $y_p(t)$ on $(0,\frac{\pi}{p}).$ Observe that our investigations approve that 
for fixed $t$ the minimum of $\Phi$ as a function of $y$ is achieved at the $y_p(t),$ i.e. $\Phi(y,t)\geq\Phi(y_p(t),t).$

Monotonicity of $y_p(t)$ follows from the monotonicity properties of $\varphi(y)$ and cosine function. This concludes the proof of our lemma and give an excuse for using the function $y_p(t).$ 
\end{proof}
\textbf{Remark}. Without the conclusion that $y_p(t)$ is well defined and differentiable, it is still possible to prove our Theorem. In fact, by inspecting our inequalities we see that we have exactly one non-trivial stationary point (trivial ones are $(0,0),$ $(y',0)$). However, then we have to check the positivity of $\Phi(y,t)$ on boundary point of rectangle $\Pi:=\{(y,t) \in [0,y']\times [0,\frac{\pi}{p}]\}.$ We will use this observation for $2\leq p\leq 4$ and the appropriate "small" rectangle at the end of the next section.
	
\section{\textbf{Proof of the Lemma 1 for $s>\csc^2\frac{\pi}{2p}$ and $p\geq 2$}}

Now, we will consider the case $s>\csc^2\frac{\pi}{2p}$. Our problem was reduced in Section 3 to proving the inequality:
\begin{equation}
\label{glavnaopsta}
\Phi(y,t)=-\cosh^{\frac{p}{s}}\big(\frac{s y}{2}\big)+C_{p,s}\big(\cosh y-\cos t\big)^{\frac{p}{2}}+D_{p,s}\cos\frac{tp}{2}\geq 0,
\end{equation} 
for $y \geq 0$ and $0 \leq t \leq \frac{2\pi}{p}.$

The constants $C_{p,s}$ and $D_{p,s}$ are chosen in such way that the local minimum is attained for $t=\frac{\pi}{p}.$ This exactly means that $C_{p,s}$ is the maximum of the function $K(y)$ defined in (\ref{Kfunkcija}), which is attained for some $\tilde{y}>0,$ while $D_{p,s}$ we find from the condition that $\frac{\partial \Phi}{\partial t}(\tilde{y},\frac{\pi}{p})=0.$ Note that $C_{p,s}>K(0)=\frac{1}{\big(1-\cos \frac{\pi}{p}\big)^{\frac{p}{2}}}=\big(\frac{s}{2}\big)^{\frac{p}{2}}.$
Also, $C_{p,s},$ by its definition, satisfies  $\frac{\partial \Phi}{\partial y}(\tilde{y},\frac{\pi}{p})=0.$ (In fact, this is how we find the value of $\tilde{y}.$)

\subsection{\textit{Main idea of our proof and reductions to inequalities for hyperbolic functions}}
We will follow the approach from the critical case. For fixed $t$ we will look for $y>0$ such that $\Phi(y,t)$ attains its minimal value. As we will see, such $y=y_p(t)$ exists and it is unique. More precisely,
$$\frac{2}{p}\frac{\partial \Phi}{\partial y}(y(t),t)=-\cosh^{\frac{p}{s}-1}\big(\frac{s y}{2}\big)\sinh\big(\frac{s y}{2}\big)+C_{p,s}\big(\cosh y-\cos t\big)^{\frac{p}{2}-1}\sinh y=0.$$
To prove that for fixed $t$ such $y_p(t)$ exists it is enough to show that the function 
$$\varphi(y)=\cosh y-\bigg(\frac{\cosh^{\frac{p}{s}-1}\big(\frac{s y}{2}\big)\sinh\big(\frac{s y}{2}\big)}{C_{p,s}\sinh y}\bigg)^{\frac{2}{p-2}}$$
is monotone increasing. We will prove this fact later.

For $t \in (0,\frac{\pi}{p})$ and $y \in (0,\tilde{y})$, we have 
\begin{align*}
\frac{2}{p}\frac{\partial \Phi}{\partial t}(y,t)&=\sin t\bigg(C_{p,s}\big(\cosh y-\cos t\big)^{\frac{p}{2}-1}-D_{p,s}\frac{\sin\frac{tp}{2}}{\sin t}\bigg)\\
&\leq \sin t\bigg(C_{p,s}\big(\cosh \tilde{y}-\cos\frac{\pi}{p}\big)^{\frac{p}{2}-1}-D_{p,s}\frac{1}{\sin\frac{\pi}{p}}\bigg)=0,
\end{align*}
since, by Lemma \ref{odnossinusa}, $\frac{1}{\sin t}\frac{\partial \Phi}{\partial t}(y,t)$ increases both in $y$ and $t$ and by the definition of $\tilde{y},$ the last expression is equal to zero. Similarly, we prove $\frac{\partial \Phi}{\partial t}(y,t)\geq 0,$ for $t\geq \frac{\pi}{p}$ and $y\geq \tilde{y}.$ In both cases, we have $\Phi(y,t)\geq \Phi(y,\frac{\pi}{p}),$ which is $\geq 0,$ by Lemma \ref{lemaoKfunkciji}. 
Now, $\varphi(\tilde{y})=\cos\frac{\pi}{p}<1$ and Lemma \ref{ocenaipsilona4b} (for $p\geq 4$) give the existence of the largest solution of $\varphi(y)=1$ with $y>\tilde{y},$ which we will denote by $y'.$ (Later we will again conclude that such $y'$ is unique! For $2\leq p< 4$ we will not consider such quantity.) Also, as in the Section 5, we see that $\varphi(y)=\cos t$ cannot have any solution for $y>y'$ (for any $t$) and $\varphi(y)=\cos t$ with $t \in (0,\frac{\pi}{p})$ makes sense only for $y \in (\tilde{y},y').$ 
After we establish the existence of the function $y_p(t),$ next step is considering $\Phi(y_p(t),t).$ We want to prove that its derivative $\frac{d \Phi}{d t} (y_p(t),t)$ is $\leq 0,$ for $0\leq t\leq\frac{\pi}{p},$ which is equivalent to
$$C_{p,s}\big(\cosh y_p(t)-\cos t\big)^{\frac{p}{2}-1}\leq D_{p,s}\frac{\sin \frac{tp}{2}}{\sin t},$$
for $0 \leq t \leq \frac{\pi}{p}.$  
Using Lemma \ref{odnossinusa} we stand at
$$ \frac{\sin \frac{tp}{2}}{\sin t} \geq E_p \cos^{\frac{p}{2}-1}t $$
for all $t \in [0,\frac{\pi}{p}]$ and $E_p=\frac{1}{\sin\frac{\pi}{p}\cos^{\frac{p}{2}-1}\frac{\pi}{p}}.$
Hence, it is enough to prove
$$D_{p,s}E_p\bigg[\cosh y-\bigg(\frac{\cosh^{\frac{p}{s}-1}\big(\frac{s y}{2}\big)\sinh\big(\frac{s y}{2}\big)}{C_{p,s}\sinh y}\bigg)^{\frac{2}{p-2}}\bigg]^{\frac{p}{2}-1}\geq \frac{\cosh^{\frac{p}{s}-1}\big(\frac{s y}{2}\big)\sinh\big(\frac{s y}{2}\big)}{\sinh y}. $$
Again, our problem of monotonicity on $t$ is equivalent to some inequality with $y$! 

From the conditions that define both $C_{p,s}$ and $D_{p,s},$ we have the identities:
\begin{equation}
\label{Cp}
C_{p,s}=\frac{\cosh^{\frac{p}{s}}\big(\frac{s \tilde{y}}{2}\big)}{\big(\cosh \tilde{y}-\cos\frac{\pi}{p}\big)^{\frac{p}{2}}}=\frac{\cosh^{\frac{p}{s}-1}\big(\frac{s \tilde{y}}{2}\big)\sinh \big(\frac{s \tilde{y}}{2}\big)}{\sinh \tilde{y}\big(\cosh\tilde{y}-\cos\frac{\pi}{p}\big)^{\frac{p}{2}-1}}
\end{equation}
and 
\begin{equation}
\label{Dp}
D_{p,s}=C_{p,s}\big(\cosh \tilde{y}-\cos\frac{\pi}{p}\big)^{\frac{p}{2}-1}\sin\frac{\pi}{p}.
\end{equation}
From these two identities, we conclude $\cosh\big(\frac{s\tilde{y}}{2}\big)=\frac{\sinh\big(\frac{s\tilde{y}}{2}\big)}{\sinh\tilde{y}}\big(\cosh\tilde{y}-\cos\frac{\pi}{p}\big)$ and $D_{p,s}E_p=C_{p,s}\big(\frac{\cosh\tilde{y}-\cos\frac{\pi}{p}}{\cos\frac{\pi}{p}}\big)^{\frac{p}{2}-1}.$

What we have to prove is now equivalent with
\begin{align*}
\cosh y &\geq \bigg(\frac{\cosh^{\frac{p}{s}-1}\big(\frac{s y}{2}\big)\sinh\big(\frac{s y}{2}\big)}{C_{p,s}\sinh y}\bigg)^{\frac{2}{p-2}}\bigg(1+\frac{\cos\frac{\pi}{p}}{\cosh\tilde{y}-\cos\frac{\pi}{p}}\bigg)\\
&\geq \cosh\tilde{y} \bigg(\frac{\cosh^{\frac{p}{s}-1}\big(\frac{s y}{2}\big)\sinh\big(\frac{s y}{2}\big)\sinh\tilde{y}}{\cosh^{\frac{p}{s}-1}\big(\frac{s\tilde{y}}{2}\big)\sinh\big(\frac{s\tilde{y}}{2}\big)\sinh y}\bigg)^{\frac{2}{p-2}}
\end{align*}
or
\begin{equation}
\label{prvaG}
G(y)\leq G(\tilde{y}),\quad \text{for}\quad \tilde{y} \leq y\leq y',
\end{equation}
where
\begin{equation}
\label{G}
G(y)=\frac{\cosh^{\frac{p}{s}-1}\big(\frac{s y}{2}\big)\sinh\big(\frac{s y}{2}\big)}{\cosh^{\frac{p}{2}-1} y\sinh y};
\end{equation}
This restriction of values for $y$ holds for $p\geq 4,$ while for $2\leq p\leq 4$ this will not be important.

In the case when $s>\csc^2\frac{\pi}{2p},$ a reduction of the main inequality to the case $0\leq t\leq \frac{\pi}{p}$ in the same way as for $s=\csc^2\frac{\pi}{2p}$ at the beginning of the Section 4 is not possible. This is caused by the fact that the maximum of the function $K(y)$ is attained for some $\tilde{y}>0!$ Thus, we need to prove that the function $y_p(t)$ is still well defined, decreasing and differentiable for some values of $t \geq \frac{\pi}{p}.$ Let us denote by $\alpha_p$ the largest number $t\leq \frac{2\pi}{p}$ such that $\varphi(y)=\cos t$ has a solution and $y''$ the smallest positive zero of $\varphi(y)=\cos\alpha_p$.
For $\frac{\pi}{p}\leq t \leq \alpha_p,$ the condition
$$C_{p,s}\big(\cosh y_p(t)-\cos t\big)^{\frac{p}{2}-1}\geq D_{p,s} \frac{\sin \frac{tp}{2}}{\sin t},$$
in sufficient for $\frac{d \Phi}{d t} (y_p(t),t)\geq 0$ (from which we conclude that $\Phi(t,y_p(t))$ has its minimum in $t=\frac{\pi}{p}$). However, the reader can easily see from the proof of Lemma \ref{odnossinusa} that for $t \in [\frac{\pi}{p},\frac{2\pi}{p}]$ the corresponding reverse estimate holds, i.e. 
$$ \frac{\sin \frac{tp}{2}}{\sin t} \leq E_p \cos^{\frac{p}{2}-1}t. $$
Proceeding similarly as in the case of $0\leq t\leq\frac{\pi}{p}$, we arrive at 
\begin{equation}
\label{drugaG}
G(y)\geq G(\tilde{y}),\quad\text{for some range of $y'$s}.
\end{equation}
In the next subsection we will see that we have to prove (\ref{drugaG}) for $y''\leq y\leq \tilde{y}$ if $p\geq 4.$

\subsection{\textit{Main inequalities for hyperbolic functions and existence of $y_p(t)$}}
Now, we will briefly analyse $G(y)$ using the conclusions from Section 5. We have proved that
$$G'(y) \leq 0 \iff g(y)=\sum_{k=1}^{+\infty}\bigg[\frac{p}{4}(s-2)^{2k+1}+\big(1-\frac{p}{4}\big)s^{2k+1}+\big(\frac{p}{4}-\frac{s}{2}\big)2^{2k+1}\bigg]\frac{y^{2k+1}}{(2k+1)!}\geq 0.
$$
Since $s > \csc^2\frac{\pi}{2p},$ we see that $p \geq 3$ implies $s> 4.$ The case of $2 \leq p \leq 3,\quad 2 \leq s \leq 4,$ we handle as $2 \leq p \leq 3$ in the critical case. For $s\geq 4$  and $2 \leq p \leq 4,$ we proceed as in the case $3 \leq p \leq 4$ in Section 5 again. Therefore, we conclude that all coefficients are positive and $g(y)>0,$ $G(y)$ is monotone decreasing and (\ref{prvaG}) and (\ref{drugaG}) for $y\geq \tilde{y}$ and $y\leq \tilde{y}$, respectively, easily follows for $2\leq p\leq 4.$ For $p>4,$ repeating the ideas from the earlier proof (in Section 5), we infer that $G(y)$ decreases till some point and then increases. 

The same argument as in the critical case confirms that, for our needs, it is enough to prove (\ref{prvaG}) for $\tilde{y}\leq y\leq y'.$ From monotonicity of $G(y),$ it follows from its proof for $y=y'.$ From $\frac{1}{\cosh y'}=1-C_{p,s}^{\frac{p}{2-p}}G(y')^{\frac{2}{p-2}}$ and $\frac{\cos\frac{\pi}{p}}{\cosh \tilde{y}}=1-C_{p,s}^{\frac{p}{2-p}}G(\tilde{y})^{\frac{2}{p-2}}$ and therefore, $G(y')\leq G(\tilde{y})$ is equivalent with 
$$\frac{1}{\cosh y'}\geq \frac{\cos\frac{\pi}{p}}{\cosh \tilde{y}},$$
which easily follows from $\tilde{y}>0$ and Lemma \ref{ocenaipsilona4b}, i.e. $y'<\cosh^{-1}\big(\frac{1}{\cos\frac{\pi}{p}}\big).$ 

However, since $G(y)$ is decreasing and then increasing, $G(y')\leq G(\tilde{y})$ implies that $\tilde{y}$ is inside the interval where $G(y)$ is decreasing (otherwise we will have $G(y') > G(\tilde{y})!$), i.e. it decreases on $[0,\tilde{y}]$.

Now, we consider monotonicity of the function 
$$
\varphi(y)=\cosh y-C_{p,s}^{-\frac{2}{p-2}}\bigg(\frac{\cosh^{\frac{p}{s}-1}\big(\frac{s y}{2}\big)\sinh\big(\frac{s y}{2}\big)}{\sinh y}\bigg)^{\frac{2}{p-2}}
=\cosh y-C_{p,s}^{-\frac{2}{p-2}}f(y)^{\frac{2}{p-2}},
$$
for $p\geq 4$ and $t \in [0,\frac{\pi}{p}].$ Since $$f'(y)=\frac{\cosh^{\frac{p}{s}-2}\big(\frac{s y}{2}\big)}{2\sinh^2 y}\bigg(p\sinh^2\big(\frac{s y}{2}\big)\sinh y+s \sinh y-\cosh y\sinh \big(s y\big)\bigg),$$
we find
$$\varphi'(y)= \sinh y-\frac{2f(y)^{\frac{2}{p-2}}}{(p-2)C_{p,s}^{\frac{2}{p-2}}\sinh y\sinh(s y)}\bigg(p\sinh^2\big(\frac{s y}{2}\big)\sinh y+s \sinh y-\cosh y\sinh \big(s y\big)\bigg).$$
Similarly as in the case of critical $s$, we infer $\varphi'(y)\geq 0,$ when the expression in the brackets is non-positive. When it is positive, from $\big(\frac{f(y)}{C_{p,s}}\big)^{\frac{2}{p-2}}\leq \frac{\cosh y}{\cosh \tilde{y}}\big(\cosh \tilde{y}-\cos\frac{\pi}{p}\big),$ for $y \in [\tilde{y},y'],$ we conclude that $\varphi'(y)\geq 0$ follows from
\begin{multline*}p \sinh(2y)\frac{\cosh(s y)-1}{2}+s \sinh(2y)-\sinh(s y)\cosh(2y)-\sinh(s y)\\
\leq \big(\frac{p}{2}-1\big)\frac{\cosh\tilde{y}}{\cosh \tilde{y}-\cos\frac{\pi}{p}}\big(\cosh(2y)-1\big)\sinh(s y),
\end{multline*}
which is equivalent to
$$-1+\frac{p}{2}\frac{\cosh y\sinh(\frac{s y}{2})}{\sinh y\cosh(\frac{s y}{2})}+\frac{s\sinh(2y)-2\sinh(s y)}{\sinh(s y)(\cosh(2y)-1)}\leq \big(\frac{p}{2}-1\big)\frac{\cosh\tilde{y}}{\cosh \tilde{y}-\cos\frac{\pi}{p}}.$$

Function $\frac{\cosh y\sinh(\frac{s y}{2})}{\sinh y\cosh(\frac{s y}{2})}$
is monotone decreasing (its derivative is equal to \\$\frac{s\sinh(2y)-2\sinh(s y)}{4\sinh^2y\cosh^2(\frac{s y}{2})}\leq 0$), hence:
$$\big(\frac{p}{2}-1\big)\frac{\cosh y\sinh(\frac{s y}{2})}{\sinh y\cosh(\frac{s y}{2})}\leq \big(\frac{p}{2}-1\big)\frac{\cosh \tilde{y}\sinh(\frac{s \tilde{y}}{2})}{\sinh \tilde{y}\cosh(\frac{s \tilde{y}}{2})}=\big(\frac{p}{2}-1\big)\frac{\cosh\tilde{y}}{\cosh \tilde{y}-\cos\frac{\pi}{p}}.$$

The last identity follows from (\ref{Cp}) and (\ref{Dp}). Note that this is an analog of (\ref{monotonost1}). Finally, the (\ref{monotonost2}) conclude the monotonicity of $\varphi(y).$
This implies that $y_p(t)$ is well defined for $t \in [0,\frac{\pi}{p}]$, monotone decreasing and differentiable, by Implicit function theorem. 

For $t \in [\frac{\pi}{p},\alpha_p]$ and $p\geq 4,$ we see that in
\begin{equation}
\label{phidrugizapis}
\varphi(y)=\cosh y\bigg[1-\bigg(\frac{\cosh^{\frac{p}{s}-1}\big(\frac{s y}{2}\big)\sinh\big(\frac{s y}{2}\big)}{C_{p,s}\cosh^{\frac{p}{2}-1} y\sinh y}\bigg)^{\frac{2}{p-2}}\bigg]=\cosh y\bigg(1-C_{p,s}^{\frac{2}{2-p}}G(y)^{\frac{2}{p-2}}\bigg)
\end{equation}
both factors are increasing and positive. Positivity of both factors can be seen from the fact that $\varphi(y)=\cos t\geq 0$ for $t \in [\frac{\pi}{p},\alpha_p].$ Also, we easily conclude that $y''$ is the only solution of $\varphi(y)=\cos\alpha_p,$ $\varphi(y)$ is increasing on $[y'',\tilde{y}]$, the existence of $y_p(t)$ for $t \in [\frac{\pi}{p},\alpha_p]$ and  $y_p([\frac{\pi}{p},\alpha_p])=[y'',\tilde{y}].$ This proves that it is enough to prove (\ref{drugaG}) for $y''\leq y\leq \tilde{y},$ although we have proved it for all $0\leq y\leq\tilde{y}.$

However, for $2\leq p\leq 4,$ we are not able to prove that $\varphi(y)$ is increasing for all $y.$ We can use the identity (\ref{phidrugizapis}) for each $y\geq\tilde{y}$ (by the same argument), while for $0\leq y\leq \tilde{y}$ we proceed in a slightly different manner. Namely, finding $y$ such that, for fixed $t$ we have $\frac{\partial \Phi}{\partial y}(y,t)=0$ we are, in fact, searching for the potential stationary points. Since (\ref{drugaG}) holds with equality just for $y=\tilde{y}$ (or $t=\frac{\pi}{p}$), we see that there is no such points in the rectangle $R:=\{(y,t) \in [0,\tilde{y}]\times[\frac{\pi}{p},\frac{2\pi}{p}]\}$ except $(y,t)=(\tilde{y},\frac{\pi}{p}).$ Hence, to prove main inequality (\ref{glavnaopsta}) we need to investigate its behaviour on the boundary of rectangle $R.$ 

For $y=\tilde{y}$ we see that $\frac{\partial \Phi}{\partial y}(\tilde{y},t)\geq \frac{\partial \Phi}{\partial y}(\tilde{y},\frac{\pi}{p})=0$ and $\Phi(\tilde{y},t)$ increases in $t;$ thus, $\Phi(y,t)$ is non-negative. 
For $y=0$ we have:
$$\Phi(0,t)=-1+C_{p,s}(1-\cos t)^{\frac{p}{2}}+D_{p,s}\cos\frac{t p}{2}.$$
Minimum in $t$ is achieved at the point where $C_{p,s}(1-\cos t)^{\frac{p}{2}-1}=\frac{D_{p,s}\sin\frac{tp}{2}}{\sin t}$, thus, we get:
\begin{align*}
\Phi(0,t)&=-1+\frac{D_{p,s}\sin\frac{tp}{2}(1-\cos t)}{\sin t}+D_{p,s}\cos\frac{tp}{2}\\
&=-1+\frac{D_{p,s}}{\sin t}\bigg(\sin\frac{tp}{2}-\sin(1-\frac{p}{2})t\bigg)\\
&=-1+D_{p,s}\frac{\sin\frac{(p-1)t}{2}}{\sin\frac{t}{2}}\geq -1+\frac{\sin\frac{\pi}{p}}{1-\cos\frac{\pi}{p}}\geq 0,
\end{align*}
since $\frac{\sin\frac{(p-1)t}{2}}{\sin\frac{t}{2}}$ is decreasing in $t$ (we consider this function for $t \in [\frac{\pi}{p},\frac{2\pi}{p}]$) and $D_{p,s}=C_{p,s}(\cosh\tilde{y}-\cos\frac{\pi}{p})^{\frac{p}{2}-1}\sin\frac{\pi}{p}\geq \big(\frac{s}{2}\big)^{\frac{p}{2}}\big(1-\cos\frac{\pi}{p}\big)^{\frac{p}{2}-1}\sin\frac{\pi}{p}=\frac{\sin\frac{\pi}{p}}{1-\cos\frac{\pi}{p}}.$ Note that the same also holds for $t \in [0,\frac{2\pi}{p}].$ 
This also excuses considering $y>0$ in our approach here for all $p\geq 2.$

For $t=\frac{\pi}{p}$, (\ref{glavnaopsta}) holds by the definition of $C_{p,s}$ and $D_{p,s}.$ Finally, if $t=\frac{2\pi}{p},$ from $\frac{\partial \Phi}{\partial t}(y,\frac{2\pi}{p})>0$ we infer that $\Phi(y,t)$ cannot have the absolute minimum in $R$ at the point of the form $(y,\frac{2\pi}{p})$; therefore, $\Phi(y,\frac{2\pi}{p})\geq 0$ for $0\leq y\leq \tilde{y}.$
This concludes the proof of (\ref{glavnaopsta}) and Theorem 1 for $p\geq 2.$.

\section{\textbf{Proof of the Lemma 1 for $s=\sec^2\frac{\pi}{2p}$ and $1<p\leq\frac{4}{3}$}}

Now, we prove (\ref{eq:NFR12}) for $1<p\leq \frac{4}{3}$ and $s=\sec^2\frac{\pi}{2p}.$ Our proof is very similar to that of $p\geq 2,$ but technically more delicate at some places. Our method works only in this range as can be seen from the proof, but the inequality that we intend to prove also does not hold for $p>\frac{4}{3}!$
Our main inequality is
\begin{equation}
\label{Phitreca}
\Phi(y,t)=-\cosh^{\frac{p}{s}}\frac{s y}{2}+\frac{(\cosh y+\cos t)^{\frac{p}{2}}}{(1+\cos\frac{\pi}{p})^{\frac{p}{2}}}-\tan\frac{\pi}{2p}\cos\frac{t p}{2} \geq 0,
\end{equation} 
for $y\geq 0$ and $t \in [0,\pi],$ as we already mention in Section 3.
Since
\begin{align*}
\frac{2}{p}\frac{\partial \Phi}{\partial t}(y,t)&=-\frac{(\cosh y+\cos t)^{\frac{p}{2}-1}\sin t}{(1+\cos\frac{\pi}{p})^{\frac{p}{2}}}+\tan\frac{\pi}{2p}\sin\frac{t p}{2}\\
 &\geq -\frac{(1+\cos t)^{\frac{p}{2}-1}\sin t}{(1+\cos\frac{\pi}{p})^{\frac{p}{2}}}+\tan\frac{\pi}{2p}\sin\frac{t p}{2},
\end{align*}
non-negativity of the last expression for $t \in [\frac{\pi}{p},\pi]$ follows from the next statement.

\begin{claim}
	The function 
$$\phi(t)=\frac{\sin\frac{t p}{2}}{\sin\frac{t}{2}\cos^{p-1}\frac{t}{2}}$$
is increasing and $\phi(t)\geq\phi(\frac{\pi}{p})$ for $t \in [\frac{\pi}{p},\pi].$
\end{claim}

\begin{proof}
We easily calculate 
$$\big(\log\phi(t)\big)'=\frac{p}{2\sin t\sin\frac{t p}{2}}\bigg[\frac{\sin(1-\frac{p}{2})t}{\sin\frac{t p}{2}}-\frac{2-p}{p}\bigg].$$
By the Claim 1 we see that $\frac{\sin(1-\frac{p}{2})t}{\sin\frac{t p}{2}}$ increases on $[0,\pi]$ and hence $\geq \frac{2-p}{p}.$ So, $\phi(t)$ increases and we are done.
\end{proof}
Now, we are reduced to the case of $t \in [0,\frac{\pi}{p}].$ For a fixed $t,$ we have
$$\frac{2}{p}\frac{\partial \Phi}{\partial y}(y,t)=\frac{(\cosh y+\cos t)^{\frac{p}{2}-1}\sinh y}{(1+\cos\frac{\pi}{p})^{\frac{p}{2}}}-\cosh^{\frac{p}{s}-1}\frac{s y}{2}\sinh \frac{s y}{2}$$
and by Lemma \ref{dobradef}, we see that there exists a unique $y_p(t)$ such that  	
\begin{equation}
\label{eq:st12}
\varphi(y)=\bigg[ \frac{\cosh^{\frac{p}{s}-1}\frac{s y}{2}\sinh \frac{s y}{2}}{\sinh y}(1+\cos\frac{\pi}{p})^{\frac{p}{2}}\bigg]^{\frac{2}{p-2}}-\cosh y=\cos t.
\end{equation}

This means that the function $\Phi(y,t)$ satisfies $\Phi(y,t)\geq\Phi(y_p(t),t);$ therefore, it is enough to prove that $\Phi(y_p(t),t)$ has its minimal value for $t=\frac{\pi}{p}.$ In fact, we will show that it is decreasing on $[0,\frac{\pi}{p}]$. From
$$\frac{2}{p}\frac{d}{d t}\Phi(y_p(t),t)=$$
$$y'_p(t)\bigg[-\cosh^{\frac{p}{s}-1}\bigg(\frac{sy_p(t)}{2}\bigg)\sinh\bigg(\frac{sy_p(t)}{2}\bigg)+\sinh y_p(t)\frac{(\cosh y_p(t)+\cos t)^{\frac{p}{2}-1}}{(1+\cos\frac{\pi}{p})^{\frac{p}{2}}}\bigg]-$$
$$\frac{(\cosh y_p(t)+\cos t)^{\frac{p}{2}-1}\sin t}{(1+\cos\frac{\pi}{p})^{\frac{p}{2}}}  +\tan\frac{\pi}{2p}\sin\frac{tp}{2}=-\frac{(\cosh y_p(t)+\cos t)^{\frac{p}{2}-1}\sin t}{(1+\cos\frac{\pi}{p})^{\frac{p}{2}}}+\tan\frac{\pi}{2p}\sin\frac{tp}{2},$$
we infer that it is a consequence of the inequality:
\begin{equation}
\label{monotonost12}
\tan\frac{\pi}{2p}\frac{\sin\frac{t p}{2}}{\sin t}  \leq \frac{(\cosh y_p(t)+\cos t)^{\frac{p}{2}-1}}{(1+\cos\frac{\pi}{p})^{\frac{p}{2}}}.
\end{equation}	
As in the case $p\geq 2,$ we will estimate the left-hand side in such a way that we can work with the appropriate range of $y'$s. The first step is given by the next
\begin{lemma}
	\label{sinus12}
	For $1<p\leq 2$ and $0 \leq t \leq \frac{\pi}{p},$ there holds the following inequality 
	$$\frac{\sin\frac{t p}{2}}{\sin t}\leq \frac{1}{\sin\frac{\pi}{p}}\bigg(\frac{c_p+\cos t}{c_p+\cos\frac{\pi}{p}}\bigg)^{\frac{p}{2}-1},$$
	where $c_p=\frac{\frac{p}{2}\sin^2\frac{\pi}{p}-1}{\cos\frac{\pi}{p}}.$
\end{lemma}
\begin{proof}
	Let us consider the function
	$$f(t):=\log\frac{\sin\frac{t p}{2}\sin\frac{\pi}{p}}{\sin t}-\big(\frac{p}{2}-1\big)\log\bigg(\frac{c_p+\cos t}{c_p+\cos\frac{\pi}{p}}\bigg).$$
	Since $f(\frac{\pi}{p})=0,$ it is enough to prove $f'(t)=\frac{p\cos\frac{t p}{2}}{2\sin\frac{t p}{2}}-\frac{\cos t}{\sin t}+(\frac{p}{2}-1)\frac{\sin t}{c_p+\cos t} \geq 0.$	
	
	From
	\begin{align*}
	F(t)&=f'(t)\sin\frac{t p}{2}\sin t(c_p+\cos t)\\
	&=\frac{p}{2}\cos\frac{t p}{2}\sin t\big(c_p+\cos t\big)-\sin\frac{t p}{2}\cos t\big(c_p+\cos t\big)+(\frac{p}{2}-1)\sin^2 t\sin\frac{t p}{2}\\
	&=\big(\frac{p}{4}-1\big)\sin\frac{t p}{2}+\frac{p}{4}\sin\big(2-\frac{p}{2}\big)t+c_p\bigg(\frac{p-2}{4}\sin\big(1+\frac{p}{2}\big)t+\frac{p+2}{4}\sin\big(1-\frac{p}{2}\big)t\bigg)
	\end{align*}
	we infer 
	$$F'(t)=\frac{p}{2}\big(1-\frac{p}{4}\big)\bigg(\cos\big(2-\frac{p}{2}\big)t-\cos\frac{t p}{2}\bigg)+c_p\frac{4-p^2}{8}\bigg(\cos\big(1-\frac{p}{2}\big)t-\cos\big(1+\frac{p}{2}\big)t\bigg)\leq 0, $$
	since 
	$$ \cos\big(2-\frac{p}{2}\big)t-\cos\frac{t p}{2}= -2 \sin t\sin(1-\frac{p}{2})t\leq 0,$$
	$$ \cos\big(1-\frac{p}{2}\big)t-\cos\big(1+\frac{p}{2}\big)t= -2 \sin t\sin\frac{t p}{2} \leq 0$$
	and $F(\frac{\pi}{p})=0$. Hence, $F(t)\geq 0$ and the Lemma \ref{sinus12} follows. 	
\end{proof}
The inequality (\ref{monotonost12}) follows from the slightly stronger
$$\frac{\tan\frac{\pi}{2p}}{\sin\frac{\pi}{p}}\bigg(\frac{c_p+\cos t}{c_p+\cos\frac{\pi}{p}}\bigg)^{\frac{p}{2}-1} \leq \frac{(\cosh y_p(t)+\cos t)^{\frac{p}{2}-1}}{(1+\cos\frac{\pi}{p})^{\frac{p}{2}}}$$
equivalent to 
$$\frac{c_p+\cos t}{c_p+\cos\frac{\pi}{p}}\geq \frac{\cosh y_p(t)+\cos t}{1+\cos\frac{\pi}{p}},$$
which, after setting $\frac{(\cosh y+\cos t)^{\frac{p}{2}-1}\sinh y}{(1+\cos\frac{\pi}{p})^{\frac{p}{2}}}-\cosh^{\frac{p}{s}-1}\frac{s y}{2}\sinh \frac{s y}{2}=0,$
reduces to
\begin{equation}
\label{eq:varphi12}
\psi(y)=c_p-\cosh y+(1-c_p)\bigg(\frac{2\cosh^{\frac{p}{s}-1}\frac{s y}{2}\sinh \frac{s y}{2}}{s \sinh y}\bigg)^{\frac{2}{p-2}} \geq 0.
\end{equation} 
It is obvious that $\psi(0)=0.$ We will prove (\ref{eq:varphi12}) for some range of $y's$ corresponding to $t \in [0,\frac{\pi}{p}]$. Our proof will be given through the following four steps:
\begin{itemize}
	\item Proving that $\psi'(y)$ is increasing till some point and then decreases,
	\item Then $\psi'(0)\geq 0,$
	\item Estimation of the $y'$ such that $\frac{(\cosh y'+1)^{\frac{p}{2}-1}\sinh y'}{(1+\cos\frac{\pi}{p})^{\frac{p}{2}}}-\cosh^{\frac{p}{s}-1}\frac{s y'}{2}\sinh \frac{s y'}{2},$
	\item Checking $\psi(y')\geq 0.$
\end{itemize}

\subsection{\textit{Step 1}}

By direct calculation we find
$$\psi'(y)=-\sinh y+\frac{2(1-c_p)}{2-p}\frac{h(y)}{\sinh y\sinh(sy)}\bigg(\frac{2\cosh^{\frac{p}{s}-1}\frac{s y}{2}\sinh \frac{s y}{2}}{s \sinh y}\bigg)^{\frac{2}{p-2}}.$$
where $h(y)=-p \sinh^2\frac{s y}{2}\sinh y-s \sinh y+\cosh y\sinh(s y).$ Let us consider 
$$ \Psi(y)= \frac{\cosh^{\alpha-1}\frac{s y}{2}\sinh^{\beta-1}\frac{s y}{2}h(y)}{\sinh^{\beta+2}y},\quad\text{with}\quad \alpha=\frac{2}{p-2}\big(\frac{p}{s}-1\big), \beta= \frac{2}{p-2}.$$
From
\begin{align*}
\Psi'(y)&=\frac{\cosh^{\alpha-2}\frac{s y}{2}\sinh^{\beta-2}\frac{s y}{2}}{2\sin^{\beta+3}y}\bigg[\sinh y\bigg(s(\alpha-1)h(y)\sinh^2\frac{s y}{2}\\
&+s(\beta-1)h(y)\cosh^2\frac{s y}{2}
+h'(y)\sinh(s y)\bigg)-(\beta+2)h(y)\cosh y\sinh(s y)\bigg]
\end{align*}
we infer that 
\begin{align*}
H(y):&=\sinh y\big(s(\alpha-1)\sinh^2\frac{s y}{2}h(y)+s(\beta-1)\cosh^2\frac{s y}{2}h(y)\\
&+h'(y)\sinh(s y)\big)-(\beta+2)\cosh y\sinh(s y)h(y)
\end{align*}
has the same sign as $\Psi'(y)$. Inserting
$h(y)=\cosh y\sinh(s y)+(\frac{p}{2}-s)\sinh y
\\-\frac{p}{2}\sinh y\cosh (s y)$ and
$h'(y)=(1-\frac{p s}{2})\sinh(s y)\sinh y+(s-\frac{p}{2})\cosh(s y)\cosh y\\
+(\frac{p}{2}-s)\cosh y$
in the formula for $H(y)$, after long, but straightforward calculation we get:
\begin{align*}
H(y)&=\frac{p(p+2)}{8(2-p)}\cosh(2s-2)y+\frac{(s-1)(2s-p)}{8}\cosh(s+2)y+\frac{p-4}{8}\cosh(2s y)\\
&+\frac{(2s-p)(3p+2s+2-ps)}{8(2-p)}\cosh(s-2)y
+\frac{(2s-p)(2p+2s-ps)}{4(p-2)}\cosh(sy)\\
&+\frac{8s^2+3p^2-2p^2 s-4ps-2p}{8(2-p)}\cosh(2y)+\frac{-8s^2-3p^2+2p^2s+4ps-6p+8}{8(2-p)}\\
&=\sum_{k=3}^{+\infty} a_k \frac{y^{2k}}{(2k)!},\quad \text{where}
\end{align*}
\begin{align*}
&a_k=\frac{(2s-p)(3p+2s+2-ps)}{8(2-p)}(s-2)^{2k}+\frac{p-4}{8}(2s)^{2k}+\frac{(s-1)(2s-p)}{8}(s+2)^{2k}+\\
&\frac{p(p+2)}{8(2-p)}(2s-2)^{2k}+\frac{(2s-p)(2p+2s-ps)}{4(p-2)}s^{2k}+\frac{8s^2+3p^2-2p^2 s-4ps-2p}{8(2-p)}2^{2k}.
\end{align*}

We will prove that several first non-zero coefficients are positive, while the other are negative. For $k=0,1,2$ coefficient is equal to zero. Third coefficient is equal to $-16s^2(s-1)(s-2)\bigg(s^2(8-14p)+s(15p^2+24-12p)+16-28p\bigg).$ To see that it is always positive, it is enough to prove that $s^2+\frac{15p^2-12p+24}{8-14p}s+2>0.$ But, $\frac{15p^2-12p+24}{8-14p}$ as a function of $p$ has the derivative $\frac{15(8+8p-7p^2)}{2(4-7p)^2}>0,$ therefore $s^2+\frac{15p^2-12p+24}{8-14p}s+2\geq s^2-\frac{9}{2}s+2>0$ for $s>4.$ (This holds for $1<p\leq \frac{4}{3}$!)
Note that we can rewrite $(2s)^{-2k}a_k$ as
\begin{equation}
\label{eq:coeff}
\frac{p(p+2)}{8(2-p)}\bigg(\frac{2s-2}{2s}\bigg)^{2k}+\frac{p-4}{8}
+\frac{8s^2+3p^2-2p^2 s-4ps-2p}{8(2-p)}\bigg(\frac{2}{2s}\bigg)^{2k}+\frac{2s-p}{16-8p}b_k,
\end{equation}
where
 $$b_k=(s-1)(2-p)\bigg(\frac{s+2}{2s}\bigg)^{2k}+(3p+2s+2-ps)\bigg(\frac{s-2}{2s}\bigg)^{2k}-2(2p+2s-ps)\bigg(\frac{1}{2}\bigg)^{2k}$$
for $k \geq 4.$ 
First three terms in (\ref{eq:coeff}) evidently decreases as a function of $k.$ If we prove that the same holds for $b_k,$ then from $\lim_{k\rightarrow +\infty}(2s)^{-2k}a_k=\frac{p-4}{8}<0$ we infer the desired conclusion. Let us, now, compare $b_k$ and $b_{k+1}.$ From 
\begin{align*}
&2^{2k}\big(b_k-b_{k+1}\big)=(s-1)(2-p)\bigg(\frac{s+2}{s}\bigg)^{2k}\frac{(s-2)(3s+2)}{4s^2}\\
&+(3p+2s+2-ps)\bigg(\frac{s-2}{s}\bigg)^{2k}\frac{(s+2)(3s-2)}{4s^2}-\frac{3}{2}\big(2p+2s-ps\big),
\end{align*}                             
we see that $b_k$ is decreasing if and only if
\begin{equation}
\label{eq:L}
f(p,s)\frac{(s-2)(3s+2)}{3s^2}\bigg(\frac{s+2}{s}\bigg)^{2k}+(1-f(p,s))\frac{(s+2)(3s-2)}{3s^2}\bigg(\frac{s-2}{s}\bigg)^{2k} \geq 1,
\end{equation}
where 
$$f(p,s)=\frac{(s-1)(2-p)}{2(2p+2s-ps)}.$$
Since $\frac{\partial }{\partial p}f(p,s)=\frac{-2(s-1)}{(2p+2s-ps)^2}<0,$ we have $f(p,s)\geq f(\frac{4}{3},s)=\frac{s-1}{2s+8}$ and 
\begin{align*}
&f(p,s)\bigg[\frac{(s-2)(3s+2)}{3s^2}\bigg(\frac{s+2}{s}\bigg)^{2k}-\frac{(s+2)(3s-2)}{3s^2}\bigg(\frac{s-2}{s}\bigg)^{2k}\bigg]\\
&\geq \frac{s-1}{2s+8}\bigg[\frac{(s-2)(3s+2)}{3s^2}\bigg(\frac{s+2}{s}\bigg)^{2k}-\frac{(s+2)(3s-2)}{3s^2}\bigg(\frac{s-2}{s}\bigg)^{2k}\bigg].
\end{align*}
Therefore, (\ref{eq:L}) follows from 
$$\frac{s-1}{2s+8}\frac{(s-2)(3s+2)}{3s^2}\bigg(\frac{s+2}{s}\bigg)^{2k}+\frac{s+9}{2s+8}\frac{(s+2)(3s-2)}{3s^2}\bigg(\frac{s-2}{s}\bigg)^{2k}\geq 1.$$
After a change of variable $x=\frac{1}{s}\leq \frac{1}{6},$ we arrive at inequality:
$$ \frac{P(x)}{24x+6}(1+2x)^{2k}+\frac{Q(x)}{24x+6}(1-2x)^{2k} \geq 1,$$
where $P(x)= (1-x)(1-2x)(3+2x)$ and $Q(x)=(1+9x)(1+2x)(3-2x).$ 

From $$P(x)(1+2x)^8+Q(x)(1-2x)^8-24x-6$$
$$=32x^2(1-x)\big(256x^8-1280x^7+1344x^6-1520x^5+48x^4+160x^3-120x^2+70x+3\big)\geq 0,$$ 
and $$P(x)(1+2x)^8+Q(x)(1-2x)^8-P(x)-Q(x)$$
$$=64x^2(1-x)\big(128x^8-640x^7+672x^6-760x^5+24x^4+80x^3-60x^2+35x+1\big)\geq 0,$$
using Jensen's inequality, we get:
\begin{align*}
&P(x)(1+2x)^{2k}+Q(x)(1-2x)^{2k}\\
&\geq (P(x)+Q(x))\bigg(\frac{P(x)}{P(x)+Q(x)}(1+2x)^8+\frac{Q(x)}{P(x)+Q(x)}(1-2x)^8\bigg)^{\frac{k}{4}}\\
&\geq P(x)(1+2x)^8+Q(x)(1-2x)^8\geq 24x+6.
\end{align*}

From Lemma \ref{koeficijentimonotonost} we conclude that $\Psi$ increases and then decreases and therefore, $\psi $ has positive and then eventually negative derivative, which concludes the statement of Step 1.

\subsection{\textit{Step 2}}

We easily see that $h(y)=\big(\frac{s}{3}+\frac{s^3}{6}-\frac{ps^2}{4}\big)y^3+ o(y^3)$ and, since $g(0+)=1,$ we have:
$$\lim_{y \rightarrow 0+}\frac{h(y)g(y)^{\frac{2}{p-2}}}{\sinh^2 y\sinh(sy)}=\frac{1}{3}-\frac{p s}{4}+\frac{s^2}{6}.$$
Now, $\psi'(0)\geq 0$ is reduced to
$$\frac{2(1-c_p)}{2-p}\bigg(\frac{1}{3}+\frac{s^2}{6}-\frac{p s}{4}\bigg)\geq 1.$$
Identity $1-c_p=\frac{2}{2-s}(1-p+\frac{p}{s})$ transforms it into  
$$\frac{2}{2-s}\bigg(1-p+\frac{p}{s}\bigg)\bigg(\frac{1}{3}+\frac{s^2}{6}-\frac{p s}{4}\bigg)+\frac{p}{2}-1\geq 0.$$
After multiplying by $s(\frac{s}{2}-1)$ we get the it is equivalent with
$$ (s-1)\bigg[\frac{p-1}{6}s^2+\bigg(-\frac{p^2}{4}+\frac{p}{2}-\frac{2}{3}\bigg)s+\frac{p}{3}\bigg] \geq 0.$$
From $-\frac{1}{4}p^2+\frac{1}{2}p-\frac{2}{3}=-\frac{1}{12}\big(3(p-1)^2+5\big)\geq -\frac{4}{9}$
we see that it is enough to prove $\frac{p-1}{6}s^2-\frac{4s}{9}+\frac{p}{3}\geq 0$, i.e.
$$p=\frac{\pi}{2\arccos\frac{1}{\sqrt{s}}} \geq \frac{3s^2+8s}{3s^2+6}.$$
We easily transform this inequality into 
$$\frac{\pi(4s-3)}{3s^2+8s} \leq \arctan \frac{1}{\sqrt{s-1}},\quad s\geq 4+2\sqrt{2}.$$
By convexity of arctangent, we get:
$$\arctan\frac{1}{\sqrt{s-1}} \geq  \frac{\sqrt{2}+1}{\sqrt{s-1}}\arctan(\sqrt{2}-1)=\frac{\pi}{8}\frac{\sqrt{2}+1}{\sqrt{s-1}}.     $$
It remains to prove
$$\frac{\sqrt{2}+1}{8}\geq \frac{(4s-3)\sqrt{s-1}}{3s^2+8s}.  $$
After substitution $x=\sqrt{s-1}\geq \sqrt{2}+1,$ we get
$$\frac{\sqrt{2}+1}{8}\geq \frac{x(4x^2+1)}{3x^4+14x^2+11}:=\kappa(x).$$
Since $\kappa(\sqrt{2}+1)=\frac{\sqrt{2}+1}{8}$ and $\kappa'(x)=\frac{-12x^6+47x^4+118x^2+11}{(3x^4+14x^2+11)^2},$ it is enough to prove that $p(t)=-12t^3+47t^2+118t+11$ is non-positive for $t\geq 3+2\sqrt{2}.$ From $p'(t)=-36t^2+94t+118$ 
and $p''(t)=2(-36t+47)<0,$ we have $p'(t)\leq p'(3+2\sqrt{2})=-212-244\sqrt{2}<0,$ hence $p(t)$ decreases and $p(t)\leq p(3+2\sqrt{2})=-8(3+5\sqrt{2})<0.$

\textbf{Remark}. This is the place where we see that our method works exactly for $1<p\leq \frac{4}{3}.$ In fact, this implies that $\psi(y)$ is decreasing in some neighbourhood of zero for $p>\frac{4}{3}$ and hence (\ref{eq:varphi12}) does not hold in this case!

\subsection{\textit{Step 3}} We will now estimate the values of $y$ such that (\ref{eq:st12}) hold for $t \in [0,\frac{\pi}{p}].$ More precisely, we show
$$\frac{2}{s}\bigg(\frac{2\cosh^{\frac{p}{s}-1}\frac{s y}{2}\sinh \frac{s y}{2}}{s \sinh y}\bigg)^{\frac{2}{p-2}}>\cosh y+1$$
for $y \geq \pi(p-1)$ for $1<p\leq \frac{5}{4}$ and $y \geq \frac{13(p-1)}{5}$
for $\frac{5}{4}\leq p \leq \frac{4}{3},$ or equivalently
$$\frac{\cosh^{\frac{p}{s}-1}\frac{s y}{2}\sinh \frac{s y}{2}}{ \sinh y}<\bigg(\frac{s}{2}\bigg)^{\frac{p}{2}}\big(\cosh y+1\big)^{\frac{p}{2}-1}.$$

We proceed as in the case $p\geq 2.$ By dividing by $\cosh^{\frac{p}{2}-1}y$ and reasoning as in Section 6 (in fact, the case of $p\leq 4$ and $s>4$ can be completely transferred to this case) we conclude that the appropriate $G(y)^{\frac{2}{p-2}}$ is increasing, while $1+\frac{1}{\cosh y}$ decreases. Hence, we are reduced to the case of $y=c(p-1),$ for corresponding $c.$ Using $\sinh\frac{s y}{2}<\cosh\frac{s y}{2},$ rephrasing inequality in terms of exponential functions, and dividing both sides by $e^{\frac{p y}{2}},$ we arrive at
$$\big(2\cos\frac{\pi}{2p}\big)^p\bigg(\frac{1+e^{-s y}}{2}\bigg)^{\frac{p}{s}}<(1+e^{-y})^{p-1}(1-e^{-y}).$$
We will prove it in the Lemma \ref{lokalizacija}. However, we first concentrate on the following auxiliary lemma.  
\begin{lemma}
	\label{pomocnalema}
	For $y=\frac{13(p-1)}{5}$ and $s=\sec^2\frac{\pi}{2p},$ the function
	$$\bigg(\frac{1+e^{-s y}}{2}\bigg)^{\frac{p}{s}}$$ is monotone decreasing on $p$
	and 
	$$\bigg(\frac{1+e^{-s y}}{2}\bigg)^{\frac{p}{s}}\leq
	\begin{cases} \frac{37}{40},\quad \frac{5}{4}\leq p\leq \frac{13}{10},\\
	\frac{9}{10}, \quad \frac{13}{10}\leq p\leq \frac{4}{3}.
	\end{cases}
	$$
\end{lemma}
\begin{proof}
	From
	$$\bigg(\frac{1+e^{-s y}}{2}\bigg)^{\frac{p}{s}}=\bigg(\frac{1+e^{-\frac{13(p-1)s}{5}}}{2}\bigg)^{\frac{1}{s(p-1)}\cdot p(p-1)}=\kappa(p)^{p(p-1)}$$
	and $\frac{d }{d p}\kappa(p)^{p(p-1)}=\kappa(p)^{p(p-1)}\bigg((2p-1)\log \kappa(p)+\frac{\kappa'(p)}{\kappa(p)}p(p-1)\bigg)$
	we conclude the result if we prove that $\kappa'(p)<0.$ But, $\kappa(p)$ is $s(p-1)-$ mean of $1$ and $e^{-\frac{13}{5}}$ and $\kappa'(p)<0$ follows from the fact that $\frac{1}{s(p-1)}$ is monotone increasing:
	$$\frac{d}{d p}\frac{\cos^2\frac{\pi}{2p}}{p-1}=\frac{\cos^2\frac{\pi}{2p}\big(\frac{\pi}{p^2}(p-1)\tan\frac{\pi}{2p}-1\big)}{(p-1)^2}>0$$
	Hence, we have 
	$$\bigg(\frac{1+e^{-s y}}{2}\bigg)^{\frac{p}{s}}\leq \bigg(\frac{1+e^{-\frac{13}{20(\sqrt{5}-1)^2}}}{2}\bigg)^{\frac{5(\sqrt{5}-1)^2}{64}}<\frac{37}{40}$$
	for $\frac{5}{4}\leq p\leq\frac{4}{3}$ and 
	\begin{align*}
	&\bigg(\frac{1+e^{-s y}}{2}\bigg)^{\frac{p}{s}}\leq  \bigg(\frac{1+e^{-\frac{39}{50}\tilde{s} }}{2}\bigg)^{\frac{13}{10\tilde{s}}}\leq\\
	& \bigg(\frac{1+e^{-\frac{156}{25}}}{2}\bigg)^{\frac{13}{80}}<\bigg(\frac{1+e^{-6}}{2}\bigg)^{\frac{13}{80}}<\frac{9}{10},
	\end{align*}
	for $\frac{13}{10}\leq p\leq \frac{4}{3}$, since $\tilde{s}=\frac{1}{\cos^2\frac{5\pi}{13}}<8.$
\end{proof}
\begin{lemma}
	\label{lokalizacija}
	There holds the inequality 
	$$\big(2\cos\frac{\pi}{2p}\big)^p\bigg(\frac{1+e^{-s y}}{2}\bigg)^{\frac{p}{s}}<(1+e^{-y})^{p-1}(1-e^{-y})$$
	with $y=\pi(p-1)$ for $1<p\leq \frac{5}{4}$ and $y=\frac{13(p-1)}{5}$ for $\frac{5}{4}<p\leq \frac{4}{3}.$
\end{lemma}
\begin{proof} We will first estimate the left-hand side of the desired inequality. The function $\frac{2p}{p-1}\cos\frac{\pi}{2p}$ is monotone decreasing on $p$ and we have
	$$2\cos\frac{\pi}{2p} \leq C \frac{p-1}{p},$$
	where
	$$	C=
	\begin{cases} \pi,\quad 1<p\leq \frac{5}{4},\\
	10\cos\frac{2\pi}{5}<\frac{31}{10}, \quad \frac{5}{4}\leq p\leq \frac{13}{10},\\
	\frac{26}{3}\cos\frac{5\pi}{13}<\frac{77}{25}, \quad \frac{13}{10}<p\leq \frac{4}{3}.
	\end{cases}
	$$	
	Using this estimate and Lemma \ref{pomocnalema} with the appropriate value of $y,$ we arrive at:
	$$ K\bigg(\frac{C(p-1)}{p}\bigg)^p-\big(1+e^{-c(p-1)}\big)^{p-1}\big(1-e^{-c(p-1)}\big)\leq 0$$
	equivalent with 
	$$f(x):=\log K+(x+1)\log\frac{C x}{x+1}-x\log(1+e^{-cx})-\log(1-e^{-cx})\leq 0.$$	
	We will prove that for the above mentioned values of the constants $K, C, c$ and a variable $x=p-1,$ it is a convex function of $x$ for $x\leq \frac{1}{2}.$ To prove
	$$f''(x)=-\frac{1}{x^2(x+1)}+\frac{c^2e^{cx}}{(e^{cx}-1)^2}+\frac{2c+ce^{cx}(2-cx)}{(e^{cx}+1)^2}>0$$
	note that the last term is evidently positive, while the positivity of the sum of the first two terms follows from:
	$$\frac{c^2x^2(x+1)}{4}\geq \bigg(\frac{1-e^{cx}}{2e^{\frac{cx}{2}}}\bigg)^2=\sinh^2\frac{cx}{2}.$$
	Now, from the elementary inequality $\sinh t\leq t\cosh t,$ we infer 
	$\frac{4}{c^2x^2}\sinh^2\frac{cx}{2} \leq \cosh^2\frac{cx}{2};$ hence, it is enough to prove
	$\cosh(cx)-2x-1\leq 0.$
	From its second derivative we easily find that this function is convex and it is enough to prove it at the ends of intervals. For $x \in [0,\frac{1}{4}]$ this follows from $\cosh\frac{\pi}{4}<\frac{3}{2},$ while on $[\frac{1}{4},\frac{1}{3}]$ we infer the conclusion from $\cosh\frac{13}{20}<\frac{3}{2}$ and $\cosh\frac{13}{15}<\frac{5}{3}$.
	
	We conclude that our function has a part-by-part convex majorant (for $p$ on the intervals $[1,\frac{5}{4}],$ $[\frac{5}{4},\frac{13}{10}]$ and $[\frac{13}{10},\frac{4}{3}]$) and since it has the maximum in one of the endpoints, it is enough to check that the values of the majorant at the endpoints are negative. Indeed, we have
	\begin{align*}
	&\lim_{p \rightarrow 1+}\bigg(\frac{\pi(p-1)}{p}\bigg)^p-\big(1+e^{-\pi(p-1)}\big)^{p-1}\big(1-e^{-\pi(p-1)}\big)=0,\\
	&\bigg(\frac{\pi}{5}\bigg)^{\frac{5}{4}}-\big(1+e^{-\frac{\pi}{4}}\big)^{\frac{1}{4}}\big(1+e^{-\frac{\pi}{4}}\big)<0,\\ &\frac{37}{40}\bigg(\frac{31}{50}\bigg)^{\frac{5}{4}}-\big(1+e^{-\frac{13}{20}}\big)^{\frac{1}{4}}\big(1+e^{-\frac{13}{20}}\big)<0,\\
	&\frac{37}{40}\bigg(\frac{93}{130}\bigg)^{\frac{13}{10}}-\big(1+e^{-\frac{39}{50}}\big)^{\frac{3}{10}}\big(1+e^{-\frac{39}{50}}\big)<0,\\ &\frac{9}{10}\bigg(\frac{231}{325}\bigg)^{\frac{13}{10}}-\big(1+e^{-\frac{39}{50}}\big)^{\frac{3}{10}}\big(1+e^{-\frac{39}{50}}\big)<0,\\
	&\frac{9}{10}\bigg(\frac{77}{100}\bigg)^{\frac{4}{3}}-\big(1+e^{-\frac{13}{15}}\big)^{\frac{1}{3}}\big(1+e^{-\frac{13}{15}}\big)<0,
	\end{align*}
 and therefore, the desired inequality follows. 	
\end{proof}	
In the following lemma we show that $y_p(t)$ is well defined and monotone increasing.
\begin{lemma}
\label{dobradef}
 The function $\varphi(y)=\frac{2}{s}\bigg[ \frac{2\cosh^{\frac{p}{s}-1}\frac{s y}{2}\sinh \frac{s y}{2}}{s\sinh y}\bigg]^{\frac{2}{p-2}}-\cosh y$ is monotone increasing on $(0,+\infty)$ and $y'\leq c(p-1)$ and $y_p(t)$ is well defined, monotone decreasing and differentiable.
\end{lemma}
\begin{proof} Since $\varphi(0)=\cos\frac{\pi}{p}$ and $\varphi(y)>1$ for $y\geq c(p-1)$ we easily conclude the existence of an $y$ such that $\varphi(y)=1.$ Finding the derivative of $\varphi,$ we have
$$\frac{\varphi'(y)}{\sinh y}=-1+\frac{4}{s(2-p)}\bigg(\frac{2}{s}\bigg)^{\frac{2}{p-2}}\Psi(y).$$	
By Step 1, $\Psi(y)$ is increasing and then decreasing, so it is enough to consider limits od $\varphi'(y)$ as $y$ tends to zero and infinity. From $\frac{2}{s}\geq 1-c_p$ we see that $\frac{\varphi'(y)}{\sinh y}\geq \frac{\psi'(y)}{\sinh y}$ and since its limit at zero is, by Step 2, non-negative, it remains to check the behaviour at $+\infty.$ Using $\sinh (ay)\sim \cosh(ay)\sim \frac{e^{a y}}{2},$ $y\rightarrow +\infty,$ we find that 
$$ \lim_{y \rightarrow +\infty}\frac{\varphi'(y)}{\sinh y}=-1+\big(s\cdot 2^{\frac{2}{s}-2}\big)^{\frac{p}{2-p}}.$$
Last expression is $\geq 0,$ since $2^{x-1}\geq x,$ for $x=\frac{2}{s} \in (0,1).$  Hence, $\varphi(y)$ is increasing and $\varphi(y)=\cos t$ has exactly one solution for $t \in [0,\frac{\pi}{p}]$. Implicit function theorem now implies the desired conclusions for $y_p(t).$  
\end{proof}

\subsection{\textit{Step 4}} We concluded that $\psi(y)$ is increasing and then decreasing and $\psi(0)\geq 0,$ for $1<p\leq \frac{4}{3}.$Therefore, it remains to prove (\ref{eq:varphi12}) for $y=c(p-1).$ 

Estimating $\psi(y)$ using Lemma 9, we get:
$$c_p-\cosh y+ (1-c_p)g(y)^{\frac{2}{p-2}}\geq c_p-\cosh y+ \frac{s(1-c_p)}{2}(\cosh y+1).$$
Hence, it is enough to prove
$$\frac{c_p+\frac{s}{2}(1-c_p)}{1-\frac{s}{2}(1-c_p)}\geq \cosh y.$$
We easily see that 
$$  \frac{c_p+\frac{s}{2}(1-c_p)}{1-\frac{s}{2}(1-c_p)}=\frac{p}{p-2}\cos\frac{\pi}{p}$$
and we are reduced to:
\begin{lemma}
	For $1<p\leq \frac{5}{4},$ there holds the inequality
	$$\frac{p}{p-2}\cos\frac{\pi}{p}\geq \cosh(\pi(p-1)), $$
	while for $\frac{5}{4}\leq p\leq \frac{4}{3}$ we have
	$$\frac{p}{p-2}\cos\frac{\pi}{p}\geq \cosh\bigg(\frac{13(p-1)}{5}\bigg). $$
\end{lemma}
\begin{proof}
	We start from well-known inequality $\cos t \geq 1-\frac{1}{2}t^2$ thus obtaining:
	$$\cos\frac{\pi}{p}=-\cos(\pi-\frac{\pi}{p})\leq -1+\frac{\pi^2(p-1)^2}{2p^2},$$
	and $\cosh \pi(p-1) \leq 1+c'\pi^2(p-1)^2,$ (by $\cosh t\leq 1+c't^2,$ with $c'=\frac{27}{50}$ and $t \in [0,\frac{4}{5}]$);
	therefore, it is enough to prove:
	$$1+c'\pi^2(p-1)^2\leq \frac{p}{p-2}\bigg(\frac{\pi^2(p-1)^2}{2p^2}-1\bigg) $$
	or, equivalently:
	$$\frac{2}{\pi^2}\geq(p-1)\big(c'(2-p)+\frac{1}{2p}\big):=\lambda(p).$$
	Since $\lambda'(p)=\frac{(5-3p)(18p^2+3p+5)}{50p^2}$, from $\lambda(\frac{5}{4})=\frac{161}{800}<\frac{2}{\pi^2}$ our inequality follows. 
	
	We used here $\mu(t)=\cosh t-1-c't^2\leq 0,$ with $c'=\frac{27}{50}$ and $t \in [0,\frac{4}{5}].$ This inequality follows from the following: $\mu'(t)=\sinh t-2c't$ and $\mu''(t)=\cosh t-2c',$ and $\mu'(t)$ decreases and then increases, hence $\mu'(\frac{4}{5})<0$ implies that $\mu'(t)>0;$ $\mu(t)$ decreases and then increases and finally $\mu(\frac{4}{5})<0$ concludes the proof. 
	
	Changing $x=\frac{p}{p-2}$ we have $\frac{p}{p-2}\cos\frac{\pi}{p}=x\sin\frac{\pi}{2x}$ and since $\frac{d}{d x}(x\sin\frac{\pi}{2x})=\\
	\cos\frac{\pi}{2x}\big(\tan\frac{\pi}{2x}-\frac{\pi}{2x}\big)>0,$ we infer that $\frac{p}{p-2}\cos\frac{\pi}{p}$ increases in $p.$ Therefore, for $\frac{5}{4}\leq p\leq \frac{13}{10}$ we have:
	$$\frac{p}{p-2}\cos\frac{\pi}{p}-\cosh \frac{13(p-1)}{5}\geq \frac{5}{3}\cos\frac{\pi}{5}-\cosh\frac{39}{50}>\frac{5(1+\sqrt{5})}{12}-\cosh\frac{4}{5}>0.$$
	Considering the function $\sigma(p)=\frac{p}{p-2}\cos\frac{\pi}{p}-\cosh \frac{13(p-1)}{5}$
	we find:
	\begin{align*}\sigma'(p)&=-\frac{2\cos\frac{\pi}{p}}{(2-p)^2}-\frac{\pi\sin\frac{\pi}{p}}{p(2-p)}-\frac{13}{3}\sinh\frac{13(p-1)}{5}\\
	&<\frac{2\cos\frac{3\pi}{13}}{(\frac{2}{3})^2}-\frac{\frac{10\pi}{13}\sin\frac{10\pi}{13}}{2-\frac{13}{10}}-\frac{13}{5}\sinh\frac{39}{50}=\frac{9\cos\frac{3\pi}{13}}{2}-\frac{100\pi\sin\frac{3\pi}{13}}{91}- \frac{13}{5}\sinh\frac{39}{50}\\
	&<\frac{27}{8}-\frac{942}{455}-\frac{13\cdot 39}{250}<0
	\end{align*}
	which gives
	$$\sigma(p)\geq \sigma(\frac{4}{3})=\sqrt{2}-\cosh\frac{13}{15}>0.$$
\end{proof}
\section{\textbf{A remark on an isoperimetric inequality for harmonic functions}}

In this section we will make some remarks on an isoperimetric inequality for harmonic functions in the unit disk. Namely, in \cite{KALAJMESTROVIC} it is proved that there is a constant $K_p$ such that
\begin{equation}
\label{isoperimetric}
\int_{\mathbb{D}} |f(z)|^{2p} d A(z) \leq K_p \bigg(\frac{1}{2\pi}\int_{0}^{2\pi}|f(e^{\imath t})|^p dt\bigg)^2
\end{equation}

for every function $f$ harmonic in the unit disk $\mathbb{D}$. Here, $f(e^{\imath t})$ denotes the value of radial limit of a function $f.$

We will prove this inequality with constant $C_p$ better than that in \cite{KALAJ.TAMS}, following the approach from \cite{MELENTIJEVICBOZIN}, where the problem earlier considered in \cite{KAYUMOVPONNUSAMYKALIRAJ} is completely solved. In \cite{HANGWANGYAN}, the authors considered similar problem in half-space setting using variational approach. 

By the basics of harmonic function theory, see \cite{PAVLOVIC.BOOK}, it can be easily seen that (\ref{isoperimetric}) is equivalent to 
$$\int_{\mathbb{D}} |P_zf(e^{\imath t})|^{2p} d A(z) \leq K_p \bigg(\frac{1}{2\pi}\int_{0}^{2\pi}|f(e^{\imath t})|^p dt\bigg)^2,$$
where $f \in L^p(\mathbb{T})$ and $P_zf(e^{\imath t})=\frac{1}{2\pi}\int_{0}^{2\pi}\frac{1-r^2}{1-2r\cos(\theta-t)+r^2}f(e^{\imath t}) dt$, $z=re^{\imath \theta}.$

Since the kernel of the integral operator $P_zf$ is positive, we infer that the constant $K_p$ has the same value for real and complex valued functions $f.$ In next lemma we give an easy observation on $K_p.$ 
\begin{lemma}
	\label{izoperimetrija}
	The best constant in the inequality (\ref{isoperimetric})
	$$K_p=C_p^p$$
	decreases as a function on $p.$
\end{lemma}
\begin{proof}
Let $q>p>1.$ From Jensen's inequality and the definition of $K_p$ we have
\begin{align*}\int_{\mathbb{D}} |f(z)|^{2q} d A(z)&\leq\int_{\mathbb{D}}\bigg(\frac{1}{2\pi}\int_{0}^{2\pi}\frac{1-r^2}{1-2r\cos(\theta-t)+r^2}|f(e^{\imath t})| dt\bigg)^{2p\cdot\frac{q}{p}}d A(z)\\
&\leq\int_{\mathbb{D}}\bigg(\frac{1}{2\pi}\int_{0}^{2\pi}\frac{1-r^2}{1-2r\cos(\theta-t)+r^2}|f(e^{\imath t})|^{\frac{q}{p}} dt\bigg)^{2p}d A(z) \\                                                           
&\leq K_p\bigg(\frac{1}{2\pi}\int_{0}^{2\pi}|f(e^{\imath t})^{\frac{q}{p}}|^p dt\bigg)^2=K_p\bigg(\frac{1}{2\pi}\int_{0}^{2\pi}|f(e^{\imath t})|^q dt\bigg)^2
\end{align*}
which gives $K_p\geq K_q.$
\end{proof}
In \cite{KALAJMESTROVIC} it is proved that  $C_2\leq \sqrt{\frac{\sqrt{2}+1}{\sqrt{2}}}$, while the maximum principle for harmonic functions gives $C_{\infty}=1$. Therefore, by our Lemma \ref{izoperimetrija} we get, for $p>2:$ 
$$C_p\leq \bigg(\frac{\sqrt{2}+1}{\sqrt{2}}\bigg)^{\frac{1}{p}}.$$
Note that the Riesz-Thorin interpolation theorem for the Poisson extension operator gives the same estimate!

For $1<p<2,$ in \cite{KALAJBAJRAMI} it is proved that
$$C_p\leq\bigg(\frac{\cos\frac{\pi}{4p}}{\cos\frac{\pi}{2p}}\bigg)^{2}. $$ We are not able to improve it further; however, we can see that the constant is the same in case of the real- and complex-valued harmonic functions.

I am grateful to David Kalaj who found several mistakes in an earlier version of the paper and to an anonymous refferee for useful remarks that improved the quality of the presentation. 

Remark. Data sharing not applicable to this article as no datasets were generated or analysed during the current study.

The (corresponding) author states that there is no conflict of interests.


\begin{thebibliography}{43}
	
\bibitem{BAERNSTEIN}
A. Baernstein, II, \text{Some sharp inequalities for conjugate functions}, Indiana Univ. Math. J. \textbf{27} (1978), 833-852.	
	
\bibitem{BANUELOSWANG}
R. Banuelos and G. Wang, \textit{Sharp inequalities for martingales with applications to the Beurling-Ahlfors transform and Riesz transforms}, Duke Math. J. \textbf{80} (1995), 575-600.	

\bibitem{BURKHOLDER}
D. L. Burkholder, \textit{Boundary value problems and sharp inequalities for martingale transforms}, Ann. Probab. \textbf{12} (1984), 647-702.

\bibitem{DAVIS}
B. Davis, \textit{On the weak type (1,1) inequality for conjugate functions}, Proc. Amer. Math. Soc. \textbf{44}, (1974), 307-311.

\bibitem{DINGGRAFAKOSZHU}
Y. Ding, L. Grafakos and K. Zhu, \textit{On the norm of the operator $aI+bH$ on $L^p(\mathbb{R})$}, Bull. Korean Math. Soc., \textbf{55}, Issue 4, (2018), 1209-1219. 
	
\bibitem{DUREN.BOOK}
P. Duren, \textit{Theory of $H^p$ spaces}, Academic Press,  New York and London, 1970.

\bibitem{ESSEN}
M. Essen, \textit{A superharmonic proof of M. Riesz conjugate function theorem}, Ark. Mat. \\ \textbf{22(1-2)} (1984), 241-249.

\bibitem{GARNETT.BOOK}
J. B. Garnett, \textit{Bounded Analytic Functions},  Springer, New York,  2007.

\bibitem{GOHBERGKRUPNIK}
I. Gohberg and N. Krupnik, \textit{Norm of the Hilbert transformation in the $L^p$ space}, Funct. Anal. Pril. \textbf{2} (1968), 91-92 [in Russian]; English transl. Funct. Anal. Appl. \textbf{2} (1968), 180-181. 

\bibitem{GK2}
I. Gohberg and N. Ya Krupnik, \textit{On the spectrum of singular integral operators in $L^p$ spaces}, Studia Math. \textbf{31} (1968), 347-362 (in Russian)

\bibitem{GK3}
I. Gohberg and N. Ya Krupnik, \textit{One-Dimensional Linear Singular Integral Equations}, Vol. II, Operator Theory: Advances and Appl., Vol.\textbf{54}, Birkh\"auser, Basel/Boston/Berlin, 1992.  

\bibitem{GRAFAKOS.MRL}
L. Grafakos,  \textit{Best bounds for the Hilbert transform on $L^p(R^1)$},  Math. Res. Lett. \textbf{4} (1997),  469-471.

\bibitem{HANGWANGYAN}
F. Hang, X. Wang and X. Yan, \textit{Sharp integral inequalities for harmonic functions}, Comm. Pure Appl. Math. \textbf{61(1)} (2008), 54-95. 

\bibitem{HORMANDER}
L. H\"ormander, \textit{Notions of convexity}, Modern Birkh\"auser Classics, Dordrecht: Springer, 2007. 

\bibitem{HKV}
B. Hollenbeck, N. J. Kalton and I. E. Verbitsky, \textit{Best constants for some opertaors associated with the Fourier and Hilbert transform}, Studia Math. \textbf{157} (2003), 237-278.

\bibitem{HV.JFA}
B. Hollenbeck and I.E. Verbitsky,  \textit{Best constants for the Riesz projection},  J. Funct. Anal. \textbf{175} (2000), 370--392.

\bibitem{HV.OTAA}
B. Hollenbeck and I.E. Verbitsky, \textit{Best constant inequalities involving the analytic and co-analytic projection},
Oper. Theory   Adv. Appl. \textbf{202} (2010), 285-295.

\bibitem{IWANIECMARTIN}
T. Iwaniec and G. Martin, \textit{Riesz transforms and related singular integrals}, J. Reine Angew. Math., \textbf{473} (1996), 25-57. 

\bibitem{JANAKIRAMAN}
P. Janakiraman, \textit{Best weak-type (p,p) constants, $1\leq p\leq 2$ for orthogonal harmonic functions and martingales}, Illinois J. Math. \textbf{48} (2004), 909-921.

\bibitem{KALAJBAJRAMI} 
D. Kalaj and E. Bajrami, \textit{On some Riesz and Carleman type inequalities for harmonic functions in the unit disk}, Comput. Methods Funct. Theory \textbf{18(2)} (2018), 295-305.

\bibitem{KALAJ.TAMS}
D. Kalaj, \textit{On Riesz type inequalities for harmonic mappings on the unit disk}, Trans. Amer. Math. Soc. \textbf{372} (2019), 4031--4051.

\bibitem{KALAJMESTROVIC}
D. Kalaj and R. Me\v{s}trovi\'{c}, \textit{An isoperimetric type inequality for harmonic functions}, J. Math. Anal. Appl. \textbf{373(2)} (2011), 439-448.

\bibitem{KAYUMOVPONNUSAMYKALIRAJ}
I. R. Kayumov, S. Ponnusamy and A. Sairam Kaliraj, \textit{Riesz-Fejer inequalities for harmonic functions}. Potential Anal. \textbf{52} (2020), 105-113. 

\bibitem{KOOSIS}
P. Koosis, \textit{"Introduction to $H^p$ spaces}, 2nd. ed. Cambridge Tracts in Math. Vol. \textbf{115}, Cambridge, University Press, Cambridge, UK, 1998. 

\bibitem{KRUPNIK}
N. Ya Krupnik, \textit{On the quotient norms of singular integral operators}, Mat, Issled. \textbf{10} (1975), 255-263 (in Russian)

\bibitem{KRVERB1}
N. Ya Krupnik and I. E. Verbitsky, \textit{The norm of the Riesz projections}, in "Linear and Complex Analysis Problem Book", Lecture Notes in Mathematics, Vol. \textbf{1043}, pp. 325-327, (1984), Springer-Verlag, Berlin/New York 

\bibitem{KRVERB2} 
N. Ya Krupnik and I. E. Verbitsky, \textit{Exact constants in theorems on the boundedness of singular operators with a weight and their applications}, Mat. Issled. \textbf{54} (1980), 21-35. (in Russian)

\bibitem{MELENTIJEVIC.PHD}
P. Melentijevi\'{c}, \textit{Estimates of gradients and operator norm estimates in harmonic function theory},
PhD thesis, Belgrade 2018.

\bibitem{MELENTIJEVICBOZIN}
P. Melentijevi\'{c} and V. Bo\v{z}in, \textit{Sharp Riesz-Fejer inequalities for harmonic Hardy spaces}, Potential Anal. \textbf{54(4)} (2021), 575-580.

\bibitem{MELENTIJEVICMARKOVIC}
P. Melentijevi\'{c} and M. Markovi\'{c}, \textit{Best constants in inequalities involving analytic and co-analytic projections and M. Riesz theorem for various function spaces}, Potential Anal (2022), https://doi.org/10.1007/s11118-022-10021-0

\bibitem{NEWMAN}
D. J. Newman, \textit{The nonexistence of projection of $L^1$ to $H^1$}, Proc. Amer. Math. Soc. \textbf{12}, (1961), 98-99.

\bibitem{OSEKOWSKI}
A. Osekowski, \textit{Sharp weak-type inequalities for Hilbert transform and Riesz projection},
Israel Journal of Mathematics \textbf{192} (2012), 429-448.

\bibitem{PAPADOPOULOS}
S. Papadopoulos, \textit{A note on the M. Riesz theorem for the conjugate functions},  Bul. Polish Acad. Sci. Math. \textbf{47} (1999), 283-288.

\bibitem{PAVLOVIC.BOOK}
M. Pavlovi\'{c}, \textit{Function theory in the unit disk}, DeGruyter, 2014.

\bibitem{PICHORIDES.STUDIA}
S.K. Pichorides, \textit{On the best values of the constants in the theorems of M. Riesz, Zygmund and Kolmogorov},
Stud.  Math. \textbf{44} (1972), 165--179.

\bibitem{RANGE}
R. M. Range, \textit{Holomorphic Functions and Integral Representations in Several Complex Variables}, Graduates Texts in Mathematics, Vol. \textbf{108}, Springer-Verlag, New York, 1986.

\bibitem{RIESZ}
M. Riesz, \textit{Sur les fonctions conjugees}, Math. Zeit. \textbf{27} (1927), 218-244.

\bibitem{RUDIN}
W. Rudin, \textit{Projections on invariant subspaces}, Proc. Amer. Math. Soc.  \textbf{13} (1962). no. 3, 429-432.

\bibitem{STEIN}
E. M. Stein, \text{Singular integrals and differentiability properties of functions}, Princeton Univ. Press, Princeton (1970)

\bibitem{TOMASZEWSKI}
B. Tomaszewski, \textit{Some sharp weak-type inequalities for holomorphic functions on the unit ball of $\mathbb{C}^n$}, Proc. Amer. Math. Soc \textbf{85} (1985), 271-274.

\bibitem{VASYUNINVOLBERG}
V. Vasyunin and A. Volberg, \textit{The Bellman function technique in harmonic analysis}, 
Cambridge studies in advanced mathematics, \textbf{186} (2020)

\bibitem{VERBITSKY.ISSLED}
I.E. Verbitsky,  \textit{Estimate of the norm of a function in a Hardy Space in terms of the norms of its real and imaginary parts},
Amer. Math. Soc. Transl.  \textbf{24} (1984),  11--15.

\bibitem{ZYGMUND}
A. Zygmund, \textit{Trigonometric series}, Vol. \textbf{2}, Cambridge University Press, London, (1968).

\end{thebibliography}
\end{document}